\documentclass{amsart}
\usepackage{amsmath}
\usepackage{amssymb}
\usepackage{amsthm}
\usepackage{mathtools}
\usepackage{tikz-cd}
\usepackage{url}

\usetikzlibrary{decorations.markings}
\usepackage{enumitem}
\usepackage[figuresright]{rotating}
\newtheoremstyle{myplain}      {10pt}{10pt}    {\itshape}{} {\scshape}{.}{.5em}{}
\newtheoremstyle{mydefinition} {10pt}{10pt}{}{}{\scshape}{.}{.5em}{}
\newtheoremstyle{myremark}     {10pt}{10pt}{}{}{\scshape}{.}{.5em}{}

\theoremstyle{myplain}
\newtheorem{thm}{Theorem}[section]
\newtheorem{lem}[thm]{Lemma}
\newtheorem{cor}[thm]{Corollary} 
\newtheorem{pro}[thm]{Proposition}
\newtheorem*{thm*}{Theorem}

\theoremstyle{mydefinition}
\newtheorem{defn}    [thm]{Definition}
\newtheorem{eg}       [thm]{Example}

\newtheorem{note}[thm]{Note}
\theoremstyle{myremark}
\newtheorem{rem}        [thm]{Remark}
\newtheorem{fact}[thm]{Fact}
\newtheorem{algo}{Algorithm} 
\newtheorem*{obsn*}{Observation}

\numberwithin{equation}{section}

\newcommand{\ZZ}{\mathbf{Z}}

\renewcommand{\Im}{\mathrm{Im}}
\DeclareMathOperator{\GL}{GL}
\DeclareMathOperator{\diag}{diag}

\DeclareMathOperator{\hgt}{ht}
\DeclareMathOperator{\Disc}{Disc}
\setlist{font=\upshape}

\title{Chinese Remainder Theorem for Cyclotomic Polynomials in $\ZZ[X]$} 
\author{Kamalakshya Mahatab}
\address{The Institute of Mathematical Sciences, Chennai}
\email{kamalakshya@imsc.res.in} 
\author{Kannappan Sampath} 
\address{Statistics and Mathematics Unit, Indian Statistical
  Institute, Bangalore-560059.}
\email{knsam.name@gmail.com}
\begin{document}
\begin{abstract}
By the Chinese remainder theorem, the canonical map \[\Psi_n:
R[X]/(X^n-1)\to \oplus_{d|n} R[X]/\Phi_d(X)\] is an isomorphism when $R$
is a field whose
characteristic does not divide $n$ and $\Phi_d$ is the $d$th cyclotomic polynomial. When $R$ is
the ring $\ZZ$ of rational integers, this map is injective but not
surjective. In this paper, we give an explicit formula for the
elementary divisors of the cokernel of $\Psi_n$ (when $R =
\mathbf{Z}$) using the prime factorisation of $n$. We also give a pictorial algorithm
using Young tableaux that takes $O(n^{3+\epsilon})$ bit operations for
any $\epsilon > 0$ to determine a basis of Smith vectors (see
Definition~\ref{defn:smith-vec}) for $\Psi_n$. In general when $R$ is
an integral
domain, we prove that the determinant of the matrix of $\Psi : R[X]/(\prod_j f_j) \to \bigoplus_j R[X]/(f_j)$ written
with respect to the standard basis is $\prod_{1 \leqslant i < j
  \leqslant n} \mathcal{R}(f_j, f_i)$, where $f_i$'s are monic
polynomials and $\mathcal{R}(f_j, f_i)$ is the resultant of $f_j$ and $f_i$.
\end{abstract}
\maketitle
\section{Introduction}
\label{sec:problem}
\subsection*{Motivation} Let $m_1, \dots, m_r$ be pairwise coprime elements in a principal ideal domain (PID) $R$, that is, for $i \neq j$, if $a \mid m_i$ and $a \mid m_j$, then, $a$ is a unit in $R$. The Chinese remainder theorem (CRT) states that, for $a_1, \dots, a_r \in R$, the system (in X)
\begin{align}
  \label{eq:sys}
  \begin{aligned}
    X &\equiv a_1 \bmod{m_1} \\
    X &\equiv a_2 \bmod{m_2} \\
    &\;\;\vdots \\
    X &\equiv a_r \bmod{m_r} 
  \end{aligned}
\end{align}
has a solution and any two solutions are congruent modulo $\prod_i
m_i$. In terms of ideals,  the natural map from the ring $R/(\prod_i
  m_i)$ to the ring $\prod_i R/ (m_i)$ is an isomorphism. 
The surjectivity of the natural map encapsulates the fact that the system \eqref{eq:sys} has a solution and the injectivity encapsulates the fact that any two solutions are congruent modulo $\prod_i m_i$.

However, such a theorem does not hold true over rings which are not PID's. For example, consider the system (in $h(X)$ over $\ZZ[X]$):
\begin{align*}
  h(X) &\equiv 1 \bmod{(X-1)} \\
  h(X) &\equiv 0 \bmod{(X+1)}.
\end{align*}
This system does not have a solution over $\ZZ[X]$: to wit, if $f_1(X), f_2(X) \in \ZZ[X]$ are such that 
\begin{align*}
  h(X) &= f_1(X) (X-1) +  1 \\
  h(X) &= f_2(X) (X+1), 
\end{align*}
then, we are led to the absurdity  $2f_2(1) = 1$. This phenemenon serves as a motivation for the questions we study in
this article.

 \subsection*{Setup} Let  $R$ be an integral domain which
is not a field, so that $R[X]$ is not a PID. Suppose that $f$ is a monic polynomial and
\[f = \prod_{i=1}^n f_i\]
where $\{f_i\}_{i=1}^n$ are pairwise coprime polynomials in $R[X]$. Consider the natural map:
\begin{align*}
  \Psi_f: R[X]/(f) &\to \bigoplus_i R[X]/(f_i) \\
  h(X) \bmod{f}  &\mapsto \bigoplus_i h(X) \bmod{f_i}. 
\end{align*}
The map becomes injective if $R$ is replaced by its field of fractions;
therefore, $\Psi_f$ is injective. However, as we have already remarked in
general, $\Psi_f$ is not surjective. As a measure of the 
failure of surjectivity, we would like to determine the cokernel $G(f)$ of the map $\Psi_f$:
\[
  \begin{tikzcd}
   0 \ar{r} & R[X]/(f) \ar{r}{\Psi_f} & \bigoplus_i R[X]/(f_i) \ar{r}{\overline{\Psi}_f} & G(f) \ar{r} & 0.
  \end{tikzcd}
\]
We would also like to understand when a given element $\alpha \in
\bigoplus_i R[X]/(f_i)$ lies in the image of $\Psi_f$. To the best of
our knowledge, it seems to us that problems of this nature have not
been explicitly studied elsewhere
in the literature. 

We shall solve the above problems when $R = \ZZ$ and $f(X) = X^n - 1$ with its factorisation $\prod_{d \mid n}\Phi_d(X)$
into cyclotomic polynomials. 

\subsection*{Results} Let us consider the map $\Psi_n$ defined by:
\begin{align*}
\Psi_n : \ZZ[X]/\langle X^n - 1 \rangle &\to \bigoplus_{d \mid n}
\ZZ[X]/\langle \Phi_d(X) \rangle \\
f(X) \bmod{(X^n - 1)}&\mapsto \bigoplus_{d \mid n} f(X) \bmod{\Phi_d(X)}.
\end{align*} The associated exact sequence is: 
\begin{equation*}
  \begin{tikzcd}
   0 \ar{r} & \ZZ[X]/(X^n-1) \ar{r}{\Psi_n} & \bigoplus_{d \mid n} \ZZ[X]/(\Phi_{d}(X)) \ar{r}{\overline{\Psi}_n} & G(n) \ar{r} & 0.
  \end{tikzcd}
\end{equation*}
 The domain and codomain of $\Psi_n$ are
free $\ZZ$-modules of the same rank and therefore, the cokernel
$G(n)$ is a finite abelian group. We endow $\ZZ[X]/(X^n-1)$ with
the basis $(1, \overline{X}, \dots, \overline{X}^{n-1})$ and
$\ZZ[X]/(\Phi_d(X))$ with the basis $(1, \overline{X}, \dots,
\overline{X}^{\phi(d)-1})$. Denote the matrix of $\Psi_n$ with respect
to this basis by $A_n$. For example, we have
\begin{equation*}
A_2 = 
\begin{pmatrix}
1 & 1 \\
1 & -1
\end{pmatrix}, \quad A_3 = 
\begin{pmatrix}
1 &  1 & 1 \\
1 & 0 & -1\\
0 & 1 & -1
\end{pmatrix}, \quad A_4 = 
\begin{pmatrix}
 1  &1  &1  &1\\
 1 &-1 & 1 &-1 \\
1  &0 &-1  &0 \\
0  &1  &0 &-1
\end{pmatrix}.
\end{equation*}
The structure of the abelian group $G(n)$ is completely determined
by the elementary divisors of $A_n$ (see for instance, \cite[Theorem 7.7]{LangAlg2002}). For example, the elementary
divisors of $A_6$ are $\{1,1,1,2,6,6\}$ and the group $G(6)$ is isomorphic to
$\ZZ/2\ZZ \oplus \ZZ/6\ZZ \oplus \ZZ/6\ZZ$. We first reduce the problem of
determining the elementary divisors of $A_n$ to that of $A_{p^e}$ for
a prime $p$ (Theorem~\ref{thm:smith-and-coprime}). For a prime $p$, the matrix $A_{p^e}$ has the following structure (Lemma~\ref{lem:recursion-for-prime-powers}):
\begin{equation}
\label{eq:A-p-e}
 A_{p^e} =  \left(\begin{matrix}
          A_{p^{e-1}} & A_{p^{e-1}} & \dots & A_{p^{e-1}} & A_{p^{e-1}} \\
          I_{p^{e-1}} & 0 & \dots & 0 & -I_{p^{e-1}} \\
          0 & I_{p^{e-1}} & \dots & 0 & -I_{p^{e-1}} \\
          \vdots &    &\ddots & & \vdots\\
          0 & 0 & \dots & I_{p^{e-1}} & -I_{p^{e-1}}
        \end{matrix}\right)
\end{equation}
with $A_{p^0}  = A_1 = (1)$; we exploit this recursive structure in determining the elementary divisors of the matrix
$A_{p^e}$.

From this approach, we deduce that (Theorem~\ref{pro:characterising-ele-divisor-ratios}), if $(e_1, \dots, e_n)$ is the
tuple of elementary divisors of $A_n$ with $e_i \mid e_{i+1}$, then
the tuple $Q_n = \left(e_1, \frac{e_2}{e_1}, \dots,
  \frac{e_n}{e_{n-1}}\right)$ of quotients  is a rearrangement of the tuple
\[(\underbrace{p_1, \dots, p_1}_{\alpha_1 \text{ times}},
\underbrace{p_2, \dots, p_2}_{\alpha_2 \text{
    times}},\dots,\underbrace{p_r,\dots,p_r}_{\alpha_r
  \text{times}},\underbrace{1,\dots,1}_{n-\sum_i\alpha_i \text{
    times}})\]
where  $n = p_1^{\alpha_1}\dots p_r^{\alpha_r}$ and
$p_i$'s are distinct primes. Moreover, the $s$th $p_i$ appears at
the index $n-\frac{n}{p_i^s}+1$ in $Q_n$. Since $|G(n)| = |\det(A_n)|
$, we have that (Corollary~\ref{cor:order-of-coker}):
\[|G(n)| = \prod_{i=1}^r p_i^{\frac {n(1-p_i^{-\alpha_i})}
  {(p_i-1)}} = \prod_{k=1}^n \gcd(k,n).\]
In Appendix~\ref{sec:det-of-an}, we prove:
\begin{equation}
\det(A_n) = \prod_{\substack{d_1, d_2 \mid n \\ 1 \leqslant d_1 < d_2
  \leqslant n}} \mathcal{R}(\Phi_{d_2}, \Phi_{d_1})  = (-1)^{n-1} \prod_{i=1}^r p_i^{\frac {n(1-p_i^{-\alpha_i})} {(p_i-1)}},
\end{equation}
where $\mathcal{R}(g_1, g_2)$ is the resultant of the polynomials
$g_1$ and $g_2$. 
More generally, if $f$ is a monic polynomial over an integral domain and if
$f = \prod_{k=1}^n f_k$
is a factorisation of $f$ into monic polynomials,
then (Theorem~\ref{thm:general-result}), 
\begin{equation}
\det(\Psi_f) = \prod_{1\leqslant i < j \leqslant n} \mathcal{R}(f_j, f_i).
\end{equation}  

We notice that the group algebra $\ZZ[G]$ over $\ZZ$ of a group $G$
isomorphic to the cyclic group $\ZZ/n\ZZ$ is $\ZZ[X]/\langle X^n -
1 \rangle$. From this perspective, the absolute value of the
determinant of $\Psi_n$ is the index of the group algebra $\ZZ[G]$ in
$\bigoplus_{j=0}^n \ZZ[X]/\Phi_{p^j}(X)$.  Raymond Ayoub and Christine
Ayoub determine this index \cite[Theorem 7(C)]{Ayoubs1969}. They
also determine a basis for this quotient $\ZZ$-module, which is then
used to determine a basis of Smith vectors (Definition~\ref{defn:smith-vec}) for the group algebra $\ZZ[G]$, in the case
$n=p^e$ for a prime $p$.

 In this paper, we carry out the
program of determining a basis of Smith vectors for a general $n$ (Section~\ref{sec:algo-for-smith-vectors}) by a pictorial
algorithm involving Young diagrams. A basis of Smith vectors for a
general $n$ can be realised as the columns of the matrix
$U_n^{-1}$ for some $U_n \in \GL_n(\ZZ)$ for which there
exists  a $V_n \in \GL_n(\ZZ)$ such that $U_nA_nV_n$ is
the Smith normal form of $A_n$
(Lemma~\ref{lem:smith-vectors-are-columns-of-AV}). 
The best known algorithm \cite[Proposition~7.20]{ASThesis} for
computing the Smith normal form of $A_n$ and the unimodular transformations
takes $O(n^{2+\theta+\epsilon})$ bit operations for any $\epsilon > 0$
where $O(n^\theta)$ is the bit complexity in multiplying two $n \times
n$ matrices over a ring $R$. In \cite{MMultVV}, it is proven that
$2 \leqslant \theta \leqslant 2.373$. Our algorithm determines a basis of Smith
vectors for a general $n$ in $O(n^{3+\epsilon})$ bit operations for any $\epsilon >
0$ without actually computing these
transformation matrices (Theorem~\ref{thm:main-algo}). The output of
the algorithm requires $O(n^{3+\epsilon})$ bits space (Lemma~\ref{lem:smith-coeff}).  

\subsection*{Framework} 
In Section~\ref{sec:smith-normal-form-an}, we compute the elementary
divisors of $A_n$. In Section~\ref{sec:basis-for-G1n}, we prove some
basic facts required in the algorithm for determining a basis of Smith vectors
for $n$ which is followed by a presentation of the algorithm in
Section~\ref{sec:algo-for-smith-vectors}. In the appendix that
follows, we compute the determinant of $A_n$ in a way that generalises
to any factorisation of a monic polynomial over a unique factorisation
domain. 
  
\section{Smith Normal Form of $A_n$}
\label{sec:smith-normal-form-an}
To establish a relationship between the Smith normal form of $A_{mn}$
and those of $A_m$ and $A_n$ for relatively prime positive integers $m$ and $n$, we begin with the following observation:  
\begin{lem}
  \label{lem:iso-of-grp-algebras}
  Given relatively prime positive integers $n$ and  $m$, the ring homomorphism
  $P_{m,n}: \ZZ[X]/(X^m-1) \otimes \ZZ[Y]/(Y^n-1) \to \ZZ[t]/(t^{mn}-1)$
  defined by:
  \[P_{m,n}(\overline{X} \otimes 1) = \overline{t}^n \text{ and } P_{m,n}(1 \otimes \overline{Y}) = \overline{t}^m\]
  is an isomorphism. Furthermore, with respect to the standard basis, the matrix of $P_{m,n}$ as a $\ZZ$-module homomorphism is a permutation matrix. 
\end{lem}
\begin{proof}
  We note that $t^{ni+mj} \equiv t^\alpha \bmod{(t^{mn}-1)}$ if and only if $ni +  mj \equiv \alpha \bmod{mn}$. Now, by the Chinese Remainder Theorem for $\ZZ$, the set 
  \[\{ni+mj:0 \leqslant i \leqslant m-1, 0 \leqslant j \leqslant n-1\}\]
  consists all the residues $\bmod \;mn$, exactly once. Thus, $P_{m,n}$
  is a bijection between the standard bases. This proves the lemma.      
\end{proof}

\begin{lem}
\label{lem:iso-of-quotients}
Suppose that $m$ and $n$ are relatively prime positive integers.
Then, the map $T_{m,n}: \ZZ[X]/\Phi_m(X) \otimes \ZZ[Y]/\Phi_n(Y) \to
\ZZ[t]/\Phi_{mn}(t)$ defined by
\begin{equation*}
\overline{X}^i \otimes \overline{Y}^j \mapsto \overline{t}^{ni+mj}
\end{equation*}
and extending $\ZZ$-linearly is a ring isomorphism. 
\end{lem}
\begin{proof}
Consider the maps:
\begin{align*}
\phi:\ZZ[X]/(\Phi_m(X)) &\to \ZZ[t]/(\Phi_{mn}(t)) &&\text{and} & {} &&
\psi: \ZZ[Y]/(\Phi_n(Y)) &\to \ZZ[t]/(\Phi_{mn}(t)) \\
\overline{X} &\mapsto \overline{t}^n && {} & {} &&\overline{Y} &\mapsto \overline{t}^m
\end{align*}  
Now, $T_{m,n}$ is composition of the canonical map $\phi
\otimes \psi$ with the identification map $\overline{f} \otimes
\overline{g} \mapsto \overline{fg}: 
\ZZ[t]/(\Phi_{mn}(t)) \otimes \ZZ[t]/(\Phi_{mn}(t)) \to
\ZZ[t]/(\Phi_{mn}(t))$.  Thus, $T_{m,n}$ is a ring homomorphism. 
 
To prove surjectivity, we show that $\overline{t} \in
\ZZ[t]/\Phi_{mn}(t)$. Indeed, since $\overline{t}$ is invertible
in $\ZZ[t]/\Phi_{mn}(t)$ and that $\gcd(m, n) = 1$, there are integers $i,
j \in \ZZ$ such that $t^{ni + mj} \equiv t \bmod{\Phi_{mn}(t)}$.

We claim this map is also injective: letting $K$ be the kernel
of the map $T_{m,n}$, the exact sequence:
\[
\begin{tikzcd}
0\ar{r} & K \ar{r} & \frac{\ZZ[X]}{\Phi_m(X)} \otimes \frac{\ZZ[Y]}{\Phi_n(Y)} \ar{r}{T_{m,n}}
&\frac{\ZZ[t]}{\Phi_{mn}(t)} \ar{r}& 0
\end{tikzcd}
\]
splits since $\ZZ[t]/\Phi_{mn}(t)$ is a free $\ZZ$-module showing:
\begin{equation}
\frac{\ZZ[X]}{\Phi_m(X)} \otimes \frac{\ZZ[Y]}{\Phi_n(Y)} \simeq K
\oplus \frac{\ZZ[t]}{\Phi_{mn}(t)}.
\end{equation}  
Being a submodule of a free module over the PID $\ZZ$, $K$
is a free $\ZZ$-module. A comparison of the rank tells us that $K$ is
of rank $0$. Thus, $K = \{0\}$, equivalently, $T_{m,n}$ is injective.  
\end{proof}
\begin{rem}
Along the lines of the proof of Lemma~\ref{lem:iso-of-quotients}, it
may be shown that for relatively prime positive integers $m$ and $n$,
the $\ZZ$-linear extension of the map
\begin{equation}
  \label{eq:pretty-iso}
\overline{X}^i \otimes \overline{Y}^j\mapsto \overline{t}^{mj+ni}  :
\frac{\ZZ[X]}{\Phi_m(X)} \otimes \frac{\ZZ[Y]}{Y^n-1} \to
\frac{\ZZ[t]}{\Phi_m(t^n)}, \quad \substack{0 \leqslant i \leqslant \phi(m)-1\\ 0
  \leqslant j \leqslant n-1}
\end{equation}
is a ring isomorphism.
\end{rem}
\subsection{Smith Equivalence of $A_m \otimes A_n$ and $A_{mn}$} For a matrix $A$ over the integers, let $S(A)$ denote the Smith normal
form of $A$ in which all the elementary divisors are
non-negative\footnote{For later purposes, we note that $\det(S(A)) = |\det(A)|$ is non-negative.}. We now state and prove one of the main results of this section: 
\begin{thm}
  \label{thm:smith-and-coprime}
  $S(A_m \otimes A_n) = S(A_{mn})$.
\end{thm}
\begin{proof}
  Consider the following diagram:
  \[ 
  \begin{tikzcd}
    \ZZ[X]/(X^m-1) \otimes \ZZ[Y]/(Y^n - 1) \ar{r}{\Psi_m \otimes \Psi_n} \ar{d}[swap]{P_{m,n}}& \bigoplus_{\substack{d_1 \mid m \\d_2 \mid n}} \ZZ[X]/\Phi_{d_1}(X) \otimes \ZZ[Y]/\Phi_{d_2}(Y) \ar{d}{T(m,n)}\\
    \ZZ[t]/(t^{mn}-1) \ar{r}{\Psi_{mn}} & \bigoplus_{d \mid  mn} \ZZ[t]/(\Phi_d(t))
  \end{tikzcd}
  \]
  The map $\Psi_m \otimes \Psi_n$ is the canonical map, defined by: 
  \[(\Psi_m \otimes \Psi_n)(X^i \otimes Y^j) = \Psi_m(X^i) \otimes \Psi_n(Y^j)\]
  and extended $\ZZ$-linearly. We shall prove that there is an
  isomorphism $T(m,n)$ that renders the diagram commutative. 

Indeed, define $T(m,n)$ by 
\begin{equation*}
 T(m,n) =\bigoplus_{\substack{d_1 \mid m \\ d_2 \mid n}} T_{d_1, d_2}
\end{equation*}
  Clearly, $T(m,n)$ is an isomorphism. As is seen by the following
  computation, $T(m,n)$ also renders the above diagram commutative:
\begin{align*}
  (\Psi_{mn} \circ P_{m,n})(\overline{X}^i \otimes \overline{Y}^j) &=
  \Psi_{mn}(\overline{t}^{ni+mj})\\
&= \oplus_{d \mid mn}  t^{ni+mj} \bmod{\Phi_{d}(t)}\\
  (T(m, n) \circ (\Psi_m \otimes \Psi_n)) (\overline{X}^i \otimes
  \overline{Y}^j) &= T(m,n)(\Psi_m(\overline{X}^i) \otimes
  \Psi_n(\overline{Y}^j)) \\ 
&= T(m,n)(\oplus_{\substack{d_1 \mid m \\ d_2 \mid n }}X^{ni} \bmod{\Phi_{d_1}(X)} \otimes Y^{mj}\bmod{\Phi_{d_2}(Y)})\\
&= \oplus_{ d \mid mn} t^{ni+mj} \bmod{\Phi_{d}(t)}.
\end{align*} 
This completes the proof. 
\end{proof}

Motivated by Theorem~\ref{thm:smith-and-coprime}, we present our strategy to determine $S(A_n)$: first determine $S(A_{p^\alpha})$ for $p^\alpha \parallel n$; since Kronecker product of diagonal matrices is a diagonal matrix, describe the smith form of a diagonal matrix; finally, use this to determine the elementary divisors of $A_n$. 

\subsection{Smith Normal Form of $A_{p^e}$ for a prime $p$} 
Let $p$ be a prime. We begin by noting that we have an explicit
formula for $\Phi_{p^e}(X)$:
\begin{equation}
\label{eq:prime-power-cyclo}
\Phi_{p^e}(X) = \sum_{i=0}^{p-1} X^{ip^{e-1}}.
\end{equation} Using this information, the following lemma determines $A_{p^e}$ recursively: 
\begin{lem}
  \label{lem:recursion-for-prime-powers}
  $A_{p^e}$ is a block matrix given by \eqref{eq:A-p-e}.
\end{lem}
\begin{proof}
  Let $A_{p^e} = (B_{ij})_{1 \leqslant i, j \leqslant p}$ where $B_{ij}$ are matrices of size $p^{e-1} \times p^{e-1}$. Since $X^{(k-1)p^{e-1}+i} \equiv X^i \bmod{\Phi_{p^j}(X)}$ when $0 \leqslant j \leqslant e-1$, $0 \leqslant i \leqslant p^{e-1}$ and $1 \leqslant k \leqslant p$, it follows that $B_{1k} = A_{p^{e-1}}$. Also, $X^i$ is itself the remainder on division by $\Phi_{p^e}(X)$ when $0 \leqslant i \leqslant \phi(p^{e}) - 1 = p^e - p^{e-1} - 1$. This shows that $(B_{ij})_{\substack{2 \leqslant i \leqslant p \\ 1 \leqslant j \leqslant p-1}}$ is an identity matrix. Finally, using 
  \[X^{p^{e-1}(p-1) + i} \equiv -X^{p^{e-1}(p-2)+i} - X^{p^{e-1}(p-3)+i} - \dots - X^{p^{e-1} + i}-X^i \bmod{\Phi_{p^e}(X)},\]
  we see that $B_{ip} = -I$ for $2 \leqslant i \leqslant p$. This completes the proof.    
\end{proof}
\begin{thm}
  \label{thm:primary-smith}
  For $e > 0$, the distinct elementary divisors of $A_{p^e}$ are $\{p^i : 0 \leqslant i \leqslant e\}$. The multiplicity of $p^i$ is $\phi(p^{e-i})$. 
\end{thm}
For $n \times n$ matrices $L$ and $M$, let us write $L \sim M$ to mean that $L$ and $M$ are Smith equivalent:
that is, $L \sim M$ if and only if there are matrices $P, Q \in \GL_n(\ZZ)$ such that $M = PLQ$. 

\begin{proof}
  We prove this by induction on $e$. \newline 

  {\bfseries The case $e = 1$.} $A_p$ is a $p \times p$ matrix of the form: 
  \[
  A_p = \left(
    \begin{matrix}
      1 & 1 & \dots & 1 & 1 \\
      1 & 0 & \dots & 0 & -1 \\
      0 & 1 & \dots & 0 & -1 \\
      \vdots & & \ddots & &\vdots \\
      0 & 0 & \dots & 1 & -1
    \end{matrix}
  \right)
  \]
  Adding all the columns to the rightmost column, we get the matrix:
  \begin{align}\label{eq:mat-Bp}
  A_p \sim B_p := \left(
    \begin{matrix}
      1 & 1 & \dots & 1 & p \\
      1 & 0 & \dots & 0 & 0 \\
      0 & 1 & \dots & 0 & 0 \\
      \vdots & & \ddots & &\vdots \\
      0 & 0 & \dots & 1 & 0
    \end{matrix}
  \right)
  \end{align}
  Therefore, we have, $|\det(A_p)| = p$, from which the theorem follows.
  \newline 

  {\bfseries The Induction Step.} Consider the matrix \eqref{eq:A-p-e}. Proceeding analogous to the $e=1$ case, we note that $A_{p^e}$ is Smith equivalent to the matrix
  \begin{align}\label{eq:mat-Bp-e}
  A_{p^e} \sim B_{p^e} := \left(
    \begin{matrix}
      A_{p^{e-1}} & A_{p^{e-1}} & \dots & A_{p^{e-1}} & pA_{p^{e-1}} \\
      I_{p^{e-1}} & 0 & \dots & 0 & 0 \\
      0 & I_{p^{e-1}} & \dots & 0 & 0 \\
      \vdots & & \ddots & &\vdots \\
      0 & 0 & \dots & I_{p^{e-1}} & 0
    \end{matrix}
  \right)
  \end{align}
  Now, performing row operations, we may obtain the following matrix, Smith equivalent to $A_{p^e}$:
  \[
  \left(
    \begin{matrix}
      0 & 0 & \dots & 0 & pA_{p^{e-1}} \\
      I_{p^{e-1}} & 0 & \dots & 0 & 0 \\
      0 & I_{p^{e-1}} & \dots & 0 & 0 \\
      \vdots & & \ddots & &\vdots \\
      0 & 0 & \dots & I_{p^{e-1}} & 0
    \end{matrix}
  \right)
  \] 
  We now interchange rows to obtain the following form: 
  \[
  \left(
    \begin{matrix}
      I_{p^{e-1}} & 0 & \dots & 0 & 0 \\
      0 & I_{p^{e-1}} & \dots & 0 & 0 \\
      \vdots & & \ddots & &\vdots \\
      0 & 0 & \dots & I_{p^{e-1}} & 0 \\
      0 & 0 & \dots & 0 & pA_{p^{e-1}}
    \end{matrix}
  \right)
 \quad \sim \quad
  \left(
    \begin{matrix}
      I_{p^{e-1}} & 0 & \dots & 0 & 0 \\
      0 & I_{p^{e-1}} & \dots & 0 & 0 \\
      \vdots & & \ddots & &\vdots \\
      0 & 0 & \dots & I_{p^{e-1}} & 0 \\
      0 & 0 & \dots & 0 & pS(A_{p^{e-1}})
    \end{matrix}
  \right)
  \]
  Since this matrix is in Smith normal form, it must be the Smith normal form of the matrix $A_{p^e}$. Now, we verify the assertions of the theorem: indeed, the elementary divisors of $A_{p^e}$ are $\{p^i: 0 \leqslant i \leqslant e\}$; the multiplicity of $p^{i+1}$ is $\phi(p^{e-1-i})$ for $0 \leqslant i \leqslant e-1$ (from the induction hypothesis) and $1$ appears $p^{e-1}(p-1) = \phi(p^e)$ times. This completes the proof.       
\end{proof}
\begin{rem}\label{rem:remarks-about-T_p_e}
  We may actually calculate the determinant of $A_{p^e}$ for $e > 0$ from the proof of Theorem~\ref{thm:primary-smith}.  Let $I(p,k)$ be the column block matrix\[I(p,k) = \left(\begin{matrix} I_{p^{k-1}} \\\vdots \\I_{p^{k-1}}\end{matrix}\right)\] 
of $p-1$ blocks. Consider the matrix $T_{p^e}$ defined as follows: 
\begin{equation}
\label{eq:T_p_e}
T_{p^e} = \begin{pmatrix}
I_{p^e-p^{e-1}} & I(p,e) \\
0      & I_{p^{e-1}}
\end{pmatrix}.
\end{equation}
Then, it is an easy computation to see that $A_{p^e} T_{p^e} = B_{p^e}$ (see \eqref{eq:mat-Bp-e}). Since $\det(T_{p^e}) = 1$ for all $p$ and $e$, we see that 
\begin{align}
\label{eq:det-recur}
\det(A_{p^e}) = \det(B_{p^e}) = (-1)^{\phi(p^e)}p^{p^{e-1}}\det(A_{p^{e-1}}).
\end{align}
This gives us a recursive formula for the determinant of $A_{p^e}$ (indeed, we know $\det(A_1) = \det((1)) = 1$). It now follows that 
\begin{equation}
\label{eq:det-formula}
\det(A_{p^e}) = (-1)^{p^e-1} p^{\sum_{k=0}^{e-1}p^{k}}. 
\end{equation}
Thus, for $e > 0$, we have that $\det(A_{p^e})$ is positive for all odd primes $p$ and negative for $p=2$.  
\end{rem}
\begin{rem}
\label{rem:totality-of-col-W}
Denote the totality of column (resp. row) operations needed to bring $A_{p^e}$ to its Smith
normal form by $V_{p^e}$ (resp. $U_{p^e}$) so that $U_{p^e}$ and $V_{p^e}$
satisfy the following:
\begin{equation}
  \label{eq:SNF-eq}
  U_{p^e} A_{p^e} V_{p^e} = S(A_{p^e}). 
\end{equation}
The matrix $V_{p^e}$ can be read off from the proof of the last
proposition to be:
\begin{equation}
  V_{p^e} = 
\begin{pmatrix}
I_{p^e - p^{e-1}} & I(p,e)V_{p^{e-1}} \\
0 & V_{p^{e-1}}
\end{pmatrix} 
= 
\begin{pmatrix}
I_{p^{e-1}} &                &              &               & V_{p^{e-1}} \\
              & I_{p^{e-1}}  &             &                & V_{p^{e-1}} \\
              &                & \ddots &                 & \vdots\\
              &                &             & I_{p^{e-1}}   & V_{p^{e-1}}\\
              &                &             &                   &V_{p^{e-1}}
\end{pmatrix}.\label{eq:V_p_e}
\end{equation}
For later purposes, we note that the following equation
sets up a recursion for the matrix $W_{p^e} :=
A_{p^e}V_{p^e}$ with $W_1 = (1)$:
\begin{equation}
  \label{eq:W_p_e}
  W_{p^e} = 
\begin{pmatrix} 
A_{p^{e-1}} & \cdots & A_{p^{e-1}} & pW_{p^{e-1}} \\
I_{p^{e-1}}  &             &                 &        0      \\
                &\ddots &                  &  \vdots   \\
                &            & I_{p^{e-1}}   &        0       
\end{pmatrix}.
\end{equation}
See Lemma~\ref{lem:smith-vectors-are-columns-of-AV} for an
interpretation of the columns of $W_{p^e}$.
\end{rem}
Let $A$ be an $n \times n$ matrix and $B$ be an $m \times m$ matrix,
then, we have: \[\det(A \otimes B) = (\det(A))^m(\det(B))^n.\] By Theorem~\ref{thm:smith-and-coprime} and \eqref{eq:det-formula} we have the following:
\begin{cor}
\label{cor:order-of-coker}
If $n = p_1^{\alpha_1} \dots p_r^{\alpha_r}$, then
\begin{align}
|\det(A_n)| &= \prod_{i=1}^r p_i^{\frac {n(1-p_i^{-\alpha_i})}
  {(p_i-1)}}
\end{align}
\end{cor}
\begin{rem}
One may cast the expression for $|\det(A_n)|$ in many different
forms. For example, by comparing the exponent of primes in both sides,
one may prove:
\begin{equation}
  \label{eq:det-many-avatars-1}
  |\det(A_n)| = \prod_{k=1}^n \gcd(k, n).
\end{equation}
In turn, this yields several nice expressions for the determinant:
\begin{equation}
  \label{eq:det-many-avatars-2}
  g(n) := \prod_{k=1}^n \gcd(k, n) = \prod_{d \mid n}
  d^{\phi\left(\frac{n}{d}\right)} = n^n \prod_{d \mid n} \frac{1}{d^{\phi(d)}}.
\end{equation}
The arithmetic properties of the function $g$ have been studied in
\cite{gcd-fn}. The author begins by observing that for a
multiplicative function $h$, the function 
\[g(n;h) := \prod_{k=1}^{n} h(\gcd(k,n))\]
satisfies a curious relationship for relatively prime positive
integers $m$ and $n$: 
\begin{equation}
  \label{eq:multiplicativity}
  g(mn;h) = g(m;h)^n g(n;h)^m.
\end{equation}
and concludes that $g(n;h)^{1/n}$ is multiplicative. However, it is now clear that underlying this curiosity is the
Kronecker product (Theorem~\ref{thm:smith-and-coprime}). 
It is shown that the Dirichlet series 
\[\sum_{n=1}^\infty \frac{\log(g(n))}{n^s}\]
converges absolutely for $\Re(s) > 2$ and equals
$-\frac{\zeta(s-1)\zeta'(s)}{\zeta(s)}$ where $\zeta(s)$ is the
Riemann's zeta function. More intricate connections between the
function $g(n)$ and the Riemann's zeta function are established (see
Corollary 4, loc.\@ cit.). It is also shown that, 
\begin{equation*}
 \text{max}\big(n^{n/v(n)}, n^{\tau(n)/(2n)}\big) \leqslant g(n) \leqslant 27 \left(\frac{\log(n)}{\omega(n)}\right)^{n\omega(n)}
\end{equation*}
where $\tau(n)$ is the number of divisors of $n$, $v(n)$ is the largest prime power divisor of $n$ and
$\omega(n)$ is the number of distinct prime factors of $n$. 
\end{rem}
Calculating the sign of this determinant turns out to be quite
tricky. We will take a different approach (see
Appendix~\ref{sec:det-of-an}) to calculate the determinant
which will also tell us the sign of $\det(A_n)$. 
\subsection{Smith Normal Form of a Diagonal Matrix} 
The next order of business is to work out the Smith normal form of a diagonal matrix, $D$: 
\[
D = \begin{pmatrix}
  a_1 \\
  & a_2 \\
  & & \ddots \\
  & & & a_n
\end{pmatrix}
\]
Notice that we may first permute the rows of $D$ so that the zero
rows of the matrix are the last few rows of $D$. If $D^\circ$ denotes
the maximum principal submatrix of $D$ whose rows are all non-zero, the Smith normal form of $D$ is, simply: 
\[
S(D) = \begin{pmatrix}
  S(D^\circ) \\
  &  0 \hphantom{D}
\end{pmatrix}
\]
Thus, we may assume that $\{a_1, \dots, a_n\}$ are all non-zero. \newline

\begin{algo}
\label{algo:diag}
Let $D = \diag(a_1, \dots, a_n)$
be a diagonal matrix with $a_i \neq 0$ for all $i$. Let $\mathcal{P}$ be the set of primes that divides at least one of the $a_i$'s. The algorithm proceeds in two steps:  
\begin{enumerate}
\item Corresponding to a prime $p_j \in \mathcal{P}$, we may associate the partition $\lambda^{(j)}$ obtained by rearranging the sequence of numbers $(\gamma_1, \dots, \gamma_i, \dots, \gamma_n)$ in weakly decreasing order, where $p_j^{\gamma_i} \parallel a_i$. Indeed, a partition associated to a prime this way has atmost $n$ non-zero parts. 
\item The elementary divisors of the matrix $D$ are now given by the formulae: 
  \begin{align*}
    e_k = \prod_{j=1}^r p_j^{\lambda^{(j)}_{n-k+1}}
  \end{align*}
  The fact that $\lambda^{(j)}$ is a sequence of weakly decreasing non-negative integers shows that 
  \[e_1 \mid \dots \mid e_n.\]
\end{enumerate}
\end{algo}
We shall find it convenient to develop a pictorial language for the algorithm. The partitions naturally suggest Young diagrams:
\begin{defn}
  [Young Diagram] 
  The Young diagram associated to a partition $\lambda = (\lambda_1, \dots, \lambda_l)$ is a left-aligned array of boxes with the $i$th row of the array containing $\lambda_i$ boxes. 
\end{defn}
For example, the Young diagram of the partition $\nu = (2, 2, 1)$ is
Figure~\ref{fig:eg-of-a-yd}. 
\begin{figure}[htbp]
  \centering
      \begin{tikzpicture}
        \matrix[matrix of nodes,nodes in empty cells, execute at empty cell = \node{\phantom{2}};] (m){
          \phantom{2}& \phantom{2}& \phantom{2}&\phantom{2}\\
          \phantom{2}& \phantom{2}& \phantom{2}&\phantom{2}\\
          \phantom{2}& \phantom{2}& \phantom{2}\\
          };
          \draw (m-1-1.north west) -- (m-1-2.north east);
          \draw (m-2-1.north west) -- (m-2-2.north east);
          \draw (m-2-1.south west) -- (m-2-2.south east);
          \draw (m-3-1.south west) -- (m-3-1.south east);

          \draw (m-1-1.north west) -- (m-3-1.south west);
          \draw (m-1-2.north west) -- (m-3-1.south east);
          \draw (m-1-3.north west) -- (m-2-2.south east);
      \end{tikzpicture}
      \caption{Young Diagram of $(2, 2, 1)$}
      \label{fig:eg-of-a-yd}
\end{figure}
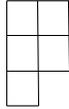
Notice that by definition, the Young diagrams of the partitions $\nu$, that of $(\nu, 0)$, $(\nu, 0, 0)$ etc.\@ are all the same.
\begin{eg}\label{egs:after-algorithm}\mbox{}
    Consider the diagonal matrix 
    \[D = \begin{pmatrix}
      6 \\
      & 4 \\
      & & 7\\
      & & & 12
    \end{pmatrix}\]
    The set $\mathcal{P}$ is therefore $\{2, 3, 7\}$. A simple calculation shows that the associated partitions are
    \begin{align*}
      2 &\leftrightarrow (2, 2, 1, 0)\\
      3 &\leftrightarrow (1, 1, 0, 0)\\
      7 &\leftrightarrow (1, 0, 0, 0)
    \end{align*}
    Therefore, the elementary divisors are: 
    \begin{align*}
      e_1 &= 2^03^07^0 = 1\\
      e_2 &= 2^13^07^0 = 2\\
      e_3 &= 2^23^17^0 = 12\\
      e_4 &= 2^23^17^1 = 84
    \end{align*}
    It may be helpful
    to draw the Young diagrams (and this will play a crucial role as
    we proceed!) on a ruled sheet of paper, see Figure~\ref{fig:ele-div-of-diagonal}. 
    \begin{figure}[htbp]
      \centering
      \begin{tikzpicture}
        \matrix[matrix of nodes,nodes in empty cells] (m){
          \phantom{2}&2&2          &[3em]&3          &[3em]&7          &[1em]84          \\
          \phantom{2}&2&2          &[3em]&3          &[3em]&\phantom{7}&[1em]12          \\
          \phantom{2}&2&\phantom{2}&[3em]&\phantom{3}&[3em]&           &[1em]\phantom{2}2\\
          \phantom{2}& &           &[3em]&           &[3em]&           &[1em]\phantom{2}1\\
        };
        \draw (m-1-1.north west) -- (m-1-8.north east);
        \foreach \i in {1,2,3,4}{
          \draw (m-\i-1.south west) -- (m-\i-8.south east);
        };
        \draw (m-1-2.north west) -- (m-3-2.south west);
        \draw (m-1-3.north west) -- (m-3-3.south west);
        \draw (m-1-3.north east) -- (m-2-3.south east);

        \draw (m-1-5.north west) -- (m-2-5.south west);
        \draw (m-1-5.north east) -- (m-2-5.south east);

        \draw (m-1-7.north west) -- (m-1-7.south west);
        \draw (m-1-7.north east) -- (m-1-7.south east);
      \end{tikzpicture}
      \caption{The Elementary Divisors of $D$}
      \label{fig:ele-div-of-diagonal}
    \end{figure}
\end{eg}
\begin{thm}
  \label{thm:correctness-of-algo}
  Algorithm~\ref{algo:diag} indeed gives the Smith normal form of the
  diagonal matrix $D=\diag(a_1, \dots, a_n)$.  
\end{thm}
\begin{proof}
 From \cite[Theorem~3.9]{BA1}, we have that a sequence $f_1, \dots,
 f_n$ satisfies 
  \begin{equation}
    \label{eq:need-to-prove}
  \prod_{j=1}^k f_j = \text{gcd of } k \times k \text{ minors of } D
  \end{equation}
  for every index $k = 1, \dots, n$ if and only if $f_1, \dots, f_n$ are the elementary
  divisors of $D$ upto units. 
  
  Let $e_1, \dots, e_n$ be the output of the algorithm. We will prove \eqref{eq:need-to-prove} with $f_j = e_j$ by backward induction on
  $k$. The case when $k = n$ is clear. Note that it suffices to prove that primes and the exponents to
  which they occur on either sides of \eqref{eq:need-to-prove} are
  equal.

  Let $\mathcal{P}=\{p_1, \dots, p_r\}$ be the set of all primes dividing atleast one of the
  $a_i$'s.  Let us set $a_k = \prod_{j=1}^r p_j^{e_{kj}}$. 
  
  Since $\gcd(\prod_j p_j^{r_j}, \prod_j p_j^{s_j}) = \prod_j
  p_j^{\min(r_j, s_j)}$, it suffices to verify the following equality for every $j$ ($1
  \leqslant j \leqslant r$):
  \begin{align*}
    \min\Big\{\sum_{i\in I} e_{ij} : I &\subseteq \{1, \dots, n\}, |I| = k\Big\} -
    \lambda^{(j)}_{n-k+1}\\
    &= \min\Big\{\sum_{i \in I} e_{ij} : I \subseteq \{1, \dots, n\}, |I| = k-1\Big\}
  \end{align*}
  for $1 \leqslant k \leqslant n$ where the notation $\lambda^{(j)}_{n-k+1}$
  is as in Algorithm~\ref{algo:diag}. But, this follows since
  $\min\Big\{\sum_{i\in I} e_{ij} : I \subseteq \{1, \dots, n\}, |I|
  = k\Big\}$ equals the sum of the first $k$ elements when, for a
  fixed $j$, the exponents $e_{ij}$'s
  are written in ascending order. That is, 
\[\min\Big\{\sum_{i\in I} e_{ij} : I \subseteq \{1, \dots, n\}, |I|
  = k\Big\} = \sum_{p=0}^{k-1} \lambda^{(j)}_{n-p}\] for every $1 \leqslant k \leqslant n$.
\end{proof}
\begin{fact}
Let  $G$ be a finite abelian group. The elementary divisors are easily computed from the primary decomposition by step 2 of the algorithm. Conversely, given its elementary divisors, the primary decomposition is the set of pairs $(p_j, \lambda^{(j)})$ obtained from step 1 of the algorithm, with these elementary divisors as the entries of a diagonal matrix. 
\end{fact}
Before we can compute the elementary divisors of $A_n$, we need to compute the partitions and primes in the Kronecker product. To do this, we recall that for the matrix $S(A_{p^\alpha})$, the set $\mathcal{P}$ is singleton $\{p\}$ and the partition associated to $p$ is, 
\begin{align*}
  \label{eq:1}
  p &\leftrightarrow (\quad \dots, \underbrace{\alpha-i, \dots, \alpha-i}_{\phi(p^i) \text{ times}}, \dots \quad), &&0 \leqslant i \leqslant \alpha
\end{align*}
where $\alpha-i$ appears $\phi(p^{i})$ times, $0 \leqslant i \leqslant \alpha$ (see Theorem~\ref{thm:primary-smith}).

Now, in the Kronecker Product, $S(A_{p_1^{n_1}}) \otimes S(A_{p_2^{n_2}})$, the set $\mathcal{P}$ of primes is $\{p_1, p_2\}$ and the associated partitions are, 
\begin{align*}
  p_1 &\leftrightarrow (\quad\dots, \underbrace{n_1-i, \dots, n_1-i}_{\phi(p_1^i)p_2^{n_2} \text{ times}}, \dots  \quad), &&0 \leqslant i \leqslant n_1\\
  p_2 &\leftrightarrow (\quad\dots, \underbrace{n_2-j, \dots, n_2-j}_{\phi(p_2^j)p_1^{n_1} \text{ times}}, \dots  \quad), &&0 \leqslant j \leqslant n_2 
\end{align*}
The following is easily seen by induction:
\begin{thm}
  \label{thm:prime-set-and-partitions-inductive-step}
  Suppose that 
  \[n = p_1^{\alpha_1} \dots p_r^{\alpha_r}\]
  is the factorisation of a positive integer $n$. Then, the set $\mathcal{P}$ of primes for the diagonal matrix 
  \[\bigotimes_j S(A_{p_j^{\alpha_j}})\]
  is the set $\{p_1, \dots, p_r\}$ and the associated partitions are 
  \begin{align*}
    p_j &\leftrightarrow (\quad\dots, \underbrace{\alpha_j-i, \dots, \alpha_j-i}_{\phi(p_j^i)n/p_j^{\alpha_j} \text{ times}}, \dots  \quad), &&0 \leqslant i \leqslant \alpha_j.
  \end{align*}
for $j = 1, \dots, r$. 
\end{thm}
Theorem~\ref{thm:prime-set-and-partitions-inductive-step} together
with Algorithm~\ref{algo:diag} completely solves the problem of determining the elementary factors of the matrix $A_n$. We illustrate this in an example: 
\begin{eg}
Let $n = 12 = 2^23^1$. Then, the set $\mathcal{P}$ of primes for the diagonal matrix $S(A_4) \otimes S(A_3)$ is $\{2, 3\}$. Associated partitions are: 
\begin{align*}
  2 &\leftrightarrow (2, 2, 2, 1, 1, 1)\\
  3 &\leftrightarrow (1, 1, 1, 1)
\end{align*}
Now, the formulae for the elementary divisors show that the elementary divisors of $A_{18}$ are $\{1, 1, 1, 1, 1, 1, 2, 2, 6, 12, 12, 12\}$. Pictorially, we have:
    \begin{figure}[htbp]
      \centering
      \begin{tikzpicture}
        \matrix[matrix of nodes,nodes in empty cells] (m){
          \phantom{2}&2          &2&[3em]&3          &[3em]&12        \\
          \phantom{2}&2          &2&[3em]&3          &[3em]&12          \\
          \phantom{2}&2          &2&[3em]&3           &[3em]&12\\
          \phantom{2}&2          & &[3em]&3           &[3em]&\phantom{1}6\\
          \phantom{2}&2          & &[3em]&           &[3em]&\phantom{1}2\\
          \phantom{2}&2         & &[3em]&\phantom{3}&[3em]&\phantom{1}2\\
          \phantom{2}&           &  &[3em]&           &[3em]&\phantom{1}1\\
          \phantom{$\displaystyle\frac12$}&\vdots&  &[3em]&          &[3em]&\phantom{$\displaystyle\frac12$}\vdots\\
          \phantom{2}&          & &[3em]&           &[3em]&\phantom{1}1\\
        };
        \draw (m-1-1.north west) -- (m-1-7.north east);
        \foreach \i in {1,...,9}{
          \draw (m-\i-1.south west) -- (m-\i-7.south east);
        };
        \draw (m-1-2.north west) -- (m-6-2.south west);
        \draw (m-1-2.north east) -- (m-6-2.south east);
       \draw (m-1-3.north east) -- (m-3-3.south east);
        \draw (m-1-5.north west) -- (m-4-5.south west);
        \draw (m-1-5.north east) -- (m-4-5.south east);
      \end{tikzpicture}
      \caption{The Elementary Divisors of $A_{12}$}
    \end{figure}
\end{eg}
\subsection{Consequences}\label{ssec:conseq-of-ele-div} 
To proceed further, we need some more notions related to a Young Diagram. The \emph{height} of a Young diagram $Y$ is the number of rows in $Y$. A cell in $Y$ is called a \emph{corner} if there is no cell to its right and there is no cell below it.

Suppose that a positive integer $n$ has the factorisation \[n = p_1^{\alpha_1} \dots p_r^{\alpha_r}\] with $\alpha_i > 0$ and $p_1 < \dots < p_r$. Let $E(A_n)$ denote the multiset of elementary divisors of $A_n$. 
\begin{pro}
\label{pro:some-ele-div-stat}
With notations as above, we have:
\begin{enumerate}
\item The multiplicity of $1$ in $E(A_n)$ is $n-\frac n {p_1}$.
  Equivalently, the rank of the cokernel of $\Psi_n$ is $\frac n{p_1}$.
\item The least integer $d > 1$ in $E(A_n)$ is $p_1$ and its multiplicity in $E(A_n)$ is 
  \begin{align*}
    \begin{cases} 
      \frac n {p_1} - \frac n {p_2},   & \text{if } p_2 < p_1^2 \\
      \frac n {p_1} - \frac n {p_1^2}, & \text{if } p_2 > p_1^2
    \end{cases}
  \end{align*}
\item The largest integer $m \in E(A_n)$ is $n$ and its multiplicity in $E(A_n)$ is \[\frac{n}{\max\left\{p_i^{\alpha_i}\mid i=1,\dots,r\right\}}.\]
\end{enumerate}
\end{pro}
In the proof of this proposition, we will make use of the general
tableaux diagram found in Figure~\ref{fig:algo-as-yd}. 
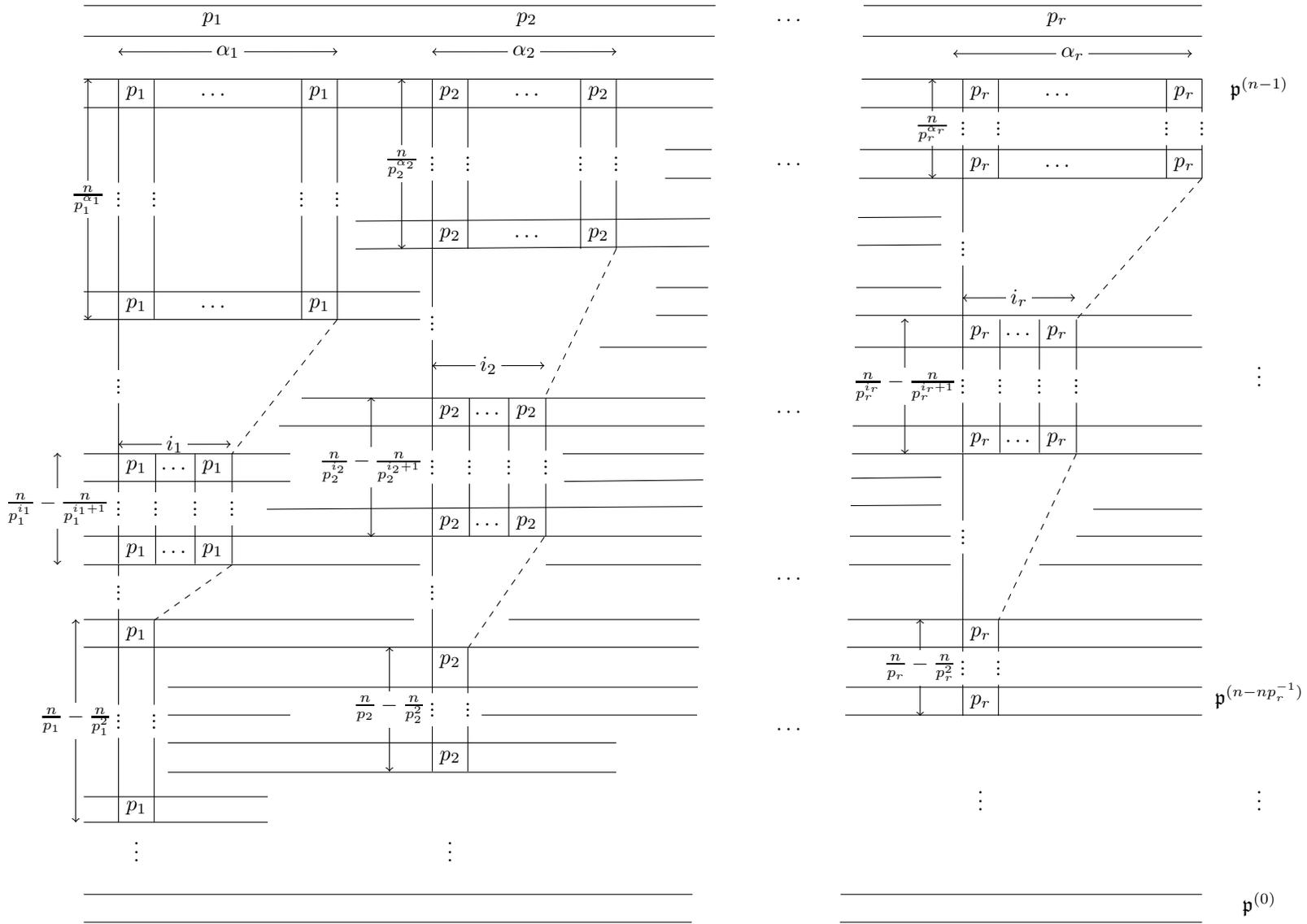
 \begin{sidewaysfigure}[htbp]
\vspace*{12cm}
\hspace*{-1cm}
 \centering
\begin{tikzpicture}[description/.style={fill=white,inner sep=2pt},
  decoration={markings,
    mark connection node=my node,
    mark=at position .5 with {\node[transform shape] (my node) {$...$};}
  } ]

  \matrix[matrix of math nodes, nodes in empty cells] (m) {
    \phantom{p_1}   & \phantom{p_1} & \phantom{p_1}      &p_1  & \phantom{p_1}  & \phantom{p_1}    & \phantom{p_1}   &[1cm]\phantom{p_1}& \phantom{p_1}  &    \phantom{p_1}&p_2  & \phantom{p_1}       &  \phantom{p_1}  &[0.5cm]&   \phantom{p_1} & \phantom{p_1}      &\dots& \phantom{p_1}  &  \phantom{p_1}   & \phantom{p_1}  &[0.5cm]& \phantom{p_1}   &   \phantom{p_1}    &p_r  &  \phantom{p_1} & \phantom{p_1}    &  \\[0.6cm]
    \phantom{p_1} &p_1&\phantom{p_1}  &\dots&\phantom{p_1}  & \phantom{p_1} &p_1&[1cm]&p_2& \phantom{p_1} &\dots&
    \phantom{p_1} &p_2&[0.5cm]& \phantom{p_1} &\phantom{p_1}  &\phantom{p_1}  & \phantom{p_1} &\phantom{p_1}  &\phantom{p_1}  &[0.5cm]&p_r&\phantom{p_1}  &\dots&
    \phantom{p_1}   &  \phantom{p_1}   &p_r &[0.1cm] \mathfrak{p}^{(n-1)}\\[0.7em]
    \phantom{p_1}  & \phantom{p_1}   &   \phantom{p_1}    &  \phantom{p_1}    & \phantom{p_1}  &   \phantom{p_1}  &   \phantom{p_1} &[1cm]& \phantom{p_1}   &  \phantom{p_1}     &    \phantom{p_1}  &   \phantom{p_1}     &\phantom{p_1}    &[0.5cm]& \phantom{p_1}   &    \phantom{p_1}   &    \phantom{p_1}  & \phantom{p_1}  &  \phantom{p_1}   &  \phantom{p_1} &[0.5cm]&  \phantom{p_1}  & \phantom{p_1}      &    \phantom{p_1}  &\phantom{p_1}   &    \phantom{p_1} &   \\
    \phantom{p_1}   & \phantom{p_1}   &  \phantom{p_1}     &   \phantom{p_1}   &\phantom{p_1}   & \phantom{p_1}    & \phantom{p_1}   &[1cm]&\phantom{p_1} &    \phantom{p_1}   &\phantom{p_1} &    \phantom{p_1}    &\phantom{p_1} &[0.5cm]&  \phantom{p_1}  & \phantom{p_1}      &   \phantom{p_1}   &  \phantom{p_1} & \phantom{p_1}    &\phantom{p_1}   &[0.5cm]&p_r   &   \phantom{p_1}    &\dots     &\phantom{p_1}   &   \phantom{p_1}  &p_r   \\[0.7em]
    \phantom{p_1}  & \phantom{p_1}   &  \phantom{p_1}     & \phantom{p_1}     & \phantom{p_1}  &   \phantom{p_1}  & \phantom{p_1}   &[1cm]&   \phantom{p_1} &    \phantom{p_1}   &     \phantom{p_1} &  \phantom{p_1}      &  \phantom{p_1}  &[0.5cm]&  \phantom{p_1}  &  \phantom{p_1}     &  \phantom{p_1}    & \phantom{p_1}  &   \phantom{p_1}  &\phantom{p_1}   &[0.5cm]&  \phantom{p_1}  &    \phantom{p_1}   &     \phantom{p_1} &\phantom{p_1}   &  \phantom{p_1}   &   \phantom{p_1} \\
    \phantom{p_1}  & \phantom{p_1}   &  \phantom{p_1}     &  \phantom{p_1}    & \phantom{p_1}  &  \phantom{p_1}   &   \phantom{p_1} &[1cm]&p_2& \phantom{p_1}      &\dots     &  \phantom{p_1}      &p_2   &[0.5cm]&  \phantom{p_1}  &  \phantom{p_1}     &   \phantom{p_1}   &  \phantom{p_1} & \phantom{p_1}    & \phantom{p_1}  &[0.5cm]&\phantom{p_1}    &  \phantom{p_1}     &    \phantom{p_1}  &  \phantom{p_1} &    \phantom{p_1} &   \\[0.7em]
    \phantom{p_1}  &  \phantom{p_1}  &  \phantom{p_1}     &  \phantom{p_1}    & \phantom{p_1}  &   \phantom{p_1}  & \phantom{p_1}   &[1cm]& \phantom{p_1}   &  \phantom{p_1}     &   \phantom{p_1}   &   \phantom{p_1}     &  \phantom{p_1}  &[0.5cm]& \phantom{p_1}   & \phantom{p_1}      & \phantom{p_1}     &  \phantom{p_1} &  \phantom{p_1}   & \phantom{p_1}  &[0.5cm]&  \phantom{p_1}  &  \phantom{p_1}     & \phantom{p_1}     & \phantom{p_1}  &    \phantom{p_1} &\phantom{p_1}    \\
    \phantom{p_1}  &p_1   &  \phantom{p_1}     &\dots &\phantom{p_1}   &   \phantom{p_1}  &p_1   &[1cm]&\phantom{p_1}    &   \phantom{p_1}    &   \phantom{p_1}   &    \phantom{p_1}    &   \phantom{p_1} &[0.5cm]&  \phantom{p_1}  &  \phantom{p_1}     &  \phantom{p_1}    &  \phantom{p_1} &\phantom{p_1}     &  \phantom{p_1} &[0.5cm]&\phantom{p_1} &    \phantom{p_1}   &\phantom{p_1} & \phantom{p_1}  &  \phantom{p_1}   &\\
    \phantom{p_1}   & \phantom{p_1}   &   \phantom{p_1}    & \phantom{p_1}     & \phantom{p_1}  &  \phantom{p_1}   &  \phantom{p_1}  &[1cm]&\phantom{p_1} &\phantom{p_1} & \phantom{p_1}  &       \phantom{p_1} & \phantom{p_1}   &[0.5cm]&  \phantom{p_1}  &   \phantom{p_1}    &   \phantom{p_1}   &\phantom{p_1}   &  \phantom{p_1}   &  \phantom{p_1} &[0.5cm]&p_r   &\dots&p_r     &  \phantom{p_1} & \phantom{p_1}    & \phantom{p_1}   \\
    \phantom{p_1}  & \phantom{p_1}   &   \phantom{p_1}    &  \phantom{p_1}    &  \phantom{p_1} &    \phantom{p_1} &  \phantom{p_1}  &[1cm]&  \phantom{p_1}  &    \phantom{p_1}   &   \phantom{p_1}   & \phantom{p_1}       &   \phantom{p_1} &[0.5cm]&\phantom{p_1}    &     \phantom{p_1}  &  \phantom{p_1}    &  \phantom{p_1} &   \phantom{p_1}  &  \phantom{p_1} &[0.5cm]&  \phantom{p_1}  &      \phantom{p_1} &   \phantom{p_1}   &  \phantom{p_1} &    \phantom{p_1} &   \phantom{p_1} &\vdots\\
    \phantom{p_1}  &\phantom{p_1} & \phantom{p_1} &  \phantom{p_1} & \phantom{p_1}  & \phantom{p_1}    & \phantom{p_1}   &[1cm]&p_2  &\dots      &p_2   &  \phantom{p_1}      &  \phantom{p_1}  &[0.5cm]&  \phantom{p_1}  &  \phantom{p_1}     &  \phantom{p_1}    & \phantom{p_1}  &\phantom{p_1}     &  \phantom{p_1} &[0.5cm]&  \phantom{p_1}  & \phantom{p_1}      &    \phantom{p_1}  &  \phantom{p_1} & \phantom{p_1}    & \phantom{p_1}   \\
    \phantom{p_1}   &  \phantom{p_1}  & \phantom{p_1}      &  \phantom{p_1}    &  \phantom{p_1} &  \phantom{p_1}   &   \phantom{p_1} &[1cm]&\phantom{p_1} &\phantom{p_1}  &\phantom{p_1}   &   \phantom{p_1}     &   \phantom{p_1} &[0.5cm]&  \phantom{p_1}  &      \phantom{p_1} &    \phantom{p_1}  &  \phantom{p_1} & \phantom{p_1}    &  \phantom{p_1} &[0.5cm]&p_r   &\dots      &p_r     & \phantom{p_1}  & \phantom{p_1}    &  \phantom{p_1}  \\
    \phantom{p_1}   &p_1   &\dots      &p_1     &  \phantom{p_1} & \phantom{p_1}    &  \phantom{p_1}  &[1cm]&   \phantom{p_1} & \phantom{p_1}      & \phantom{p_1}     &  \phantom{p_1}      &  \phantom{p_1}  &[0.5cm]& \phantom{p_1}   &    \phantom{p_1}   & \phantom{p_1}     &  \phantom{p_1} & \phantom{p_1}    & \phantom{p_1}  &[0.5cm]&\phantom{p_1} & \phantom{p_1} & \phantom{p_1} &  \phantom{p_1} &  \phantom{p_1}   & \phantom{p_1}   \\
    \phantom{p_1}   &  \phantom{p_1}  & \phantom{p_1}      &   \phantom{p_1}   &\phantom{p_1}   &   \phantom{p_1}  &   \phantom{p_1} &[1cm]& \phantom{p_1}   & \phantom{p_1}      &    \phantom{p_1}  & \phantom{p_1}       &  \phantom{p_1}  &[0.5cm]&   \phantom{p_1} & \phantom{p_1}      &  \phantom{p_1}    &  \phantom{p_1} & \phantom{p_1}    &\phantom{p_1}   &[0.5cm]&  \phantom{p_1}  &   \phantom{p_1}    &    \phantom{p_1}  & \phantom{p_1}  & \phantom{p_1}    & \phantom{p_1}   \\
    \phantom{p_1}   &\phantom{p_1} &\phantom{p_1} &  \phantom{p_1} &  \phantom{p_1} & \phantom{p_1}    &  \phantom{p_1}  &[1cm]&p_2   &\dots      &p_2     &   \phantom{p_1}     &  \phantom{p_1}  &[0.5cm]&   \phantom{p_1} & \phantom{p_1}      & \phantom{p_1}     &\phantom{p_1}   & \phantom{p_1}    &  \phantom{p_1} &[0.5cm]&\phantom{p_1}    &    \phantom{p_1}   &   \phantom{p_1}   & \phantom{p_1}  &    \phantom{p_1} & \phantom{p_1}   \\
    \phantom{p_1}   &p_1   &\dots      &p_1     &  \phantom{p_1} & \phantom{p_1}    & \phantom{p_1}   &[1cm]&\phantom{p_1} &  \phantom{p_1}     &  \phantom{p_1}    &   \phantom{p_1}     &   \phantom{p_1} &[0.5cm]& \phantom{p_1}   &  \phantom{p_1}     &  \phantom{p_1}    &\phantom{p_1}   &  \phantom{p_1}   & \phantom{p_1}  &[0.5cm]&\phantom{p_1} &\phantom{p_1}  & \phantom{p_1} & \phantom{p_1}  &   \phantom{p_1}  & \phantom{p_1}   \\
    \phantom{p_1}   & \phantom{p_1}   & \phantom{p_1}      & \phantom{p_1}     & \phantom{p_1}  &  \phantom{p_1}   & \phantom{p_1}   &[1cm]&  \phantom{p_1}  &    \phantom{p_1}   &\phantom{p_1}      &   \phantom{p_1}     &   \phantom{p_1} &[0.5cm]&\phantom{p_1}    &    \phantom{p_1}   &  \phantom{p_1}    & \phantom{p_1}  &    \phantom{p_1} & \phantom{p_1}  &[0.5cm]&   \phantom{p_1} &      \phantom{p_1} & \phantom{p_1}     & \phantom{p_1}  &  \phantom{p_1}   &  \phantom{p_1}  \\
    \phantom{p_1}   &  \phantom{p_1}  &  \phantom{p_1}     &    \phantom{p_1}  & \phantom{p_1}  & \phantom{p_1}    & \phantom{p_1}   &[1cm]& \phantom{p_1}   &      \phantom{p_1} & \phantom{p_1}     &  \phantom{p_1}      & \phantom{p_1}   &[0.5cm]& \phantom{p_1}   &   \phantom{p_1}    &  \phantom{p_1}    &  \phantom{p_1} &  \phantom{p_1}   & \phantom{p_1}  &[0.5cm]&  \phantom{p_1}  &  \phantom{p_1}     & \phantom{p_1}     &\phantom{p_1}   &   \phantom{p_1}  &   \phantom{p_1} \\
    \phantom{p_1}   &p_1   &    \phantom{p_1}   &   \phantom{p_1}   & \phantom{p_1}  & \phantom{p_1}    &   \phantom{p_1} &[1cm]&\phantom{p_1} &   \phantom{p_1}    &  \phantom{p_1}    &    \phantom{p_1}    & \phantom{p_1}   &[0.5cm]& \phantom{p_1}   &    \phantom{p_1}   &  \phantom{p_1}    & \phantom{p_1}  &  \phantom{p_1}   &  \phantom{p_1} &[0.5cm]&p_r   &     \phantom{p_1}  & \phantom{p_1}     & \phantom{p_1}  & \phantom{p_1}    & \phantom{p_1}   \\
    \phantom{p_1}   &\phantom{p_1} & \phantom{p_1}      &   \phantom{p_1}   &\phantom{p_1}   & \phantom{p_1}    &   \phantom{p_1} &[1cm]&p_2   &   \phantom{p_1}    &  \phantom{p_1}    &  \phantom{p_1}      &  \phantom{p_1}  &[0.5cm]& \phantom{p_1}   & \phantom{p_1}      &  \phantom{p_1}    & \phantom{p_1}  &   \phantom{p_1}  &\phantom{p_1}   &[0.5cm]&  \phantom{p_1}  &    \phantom{p_1}   &\phantom{p_1}      &  \phantom{p_1} &    \phantom{p_1} & \phantom{p_1}   \\
    \phantom{p_1}   &  \phantom{p_1}  &   \phantom{p_1}    & \phantom{p_1}     &\phantom{p_1}   &\phantom{p_1}     &  \phantom{p_1}  &[1cm]&   \phantom{p_1} &   \phantom{p_1}    &   \phantom{p_1}   & \phantom{p_1}       & \phantom{p_1}   &[0.5cm]& \phantom{p_1}   &     \phantom{p_1}  &   \phantom{p_1}   &\phantom{p_1}   & \phantom{p_1}    &  \phantom{p_1} &[0.5cm]&p_r&    \phantom{p_1}   &     \phantom{p_1} & \phantom{p_1}  & \phantom{p_1}    &\phantom{p_1} &[0.1cm]\mathfrak{p}^{(n-np_r^{-1})}   \\
    \phantom{p_1}   & \phantom{p_1}   &  \phantom{p_1}     &    \phantom{p_1}  &  \phantom{p_1} & \phantom{p_1}    &  \phantom{p_1}  &[1cm]&   \phantom{p_1} &  \phantom{p_1}     &    \phantom{p_1}  &   \phantom{p_1}     &  \phantom{p_1}  &[0.5cm]& \phantom{p_1}   &   \phantom{p_1}    &\phantom{p_1}      &  \phantom{p_1} &    \phantom{p_1} &  \phantom{p_1} &[0.5cm]&\phantom{p_1}    &     \phantom{p_1}  &  \phantom{p_1}    & \phantom{p_1}  &    \phantom{p_1} & \phantom{p_1}   \\
    &\phantom{p_1} &  \phantom{p_1}     & \phantom{p_1}     &  \phantom{p_1} &   \phantom{p_1}  &  \phantom{p_1}  &[1cm]&p_2   &  \phantom{p_1}     &  \phantom{p_1}    &       \phantom{p_1} &\phantom{p_1}    &[0.5cm]&\phantom{p_1}    &      \phantom{p_1} &   \phantom{p_1}   &  \phantom{p_1} &    \phantom{p_1} &  \phantom{p_1} &[0.5cm]&\phantom{p_1}    &   \phantom{p_1}  &   \phantom{p_1}   & \phantom{p_1}  &  \phantom{p_1}   &   \\
    \phantom{p_1} &p_1 &\phantom{p_1}  &\phantom{p_1}  & \phantom{p_1} & \phantom{p_1} & \phantom{p_1} &[1cm]&\phantom{p_1}  & \phantom{p_1} & \phantom{p_1} & \phantom{p_1} &
    \phantom{p_1} &[0.5cm]& \phantom{p_1} & \phantom{p_1} & \phantom{p_1} & \phantom{p_1} &\phantom{p_1}  & \phantom{p_1} &[0.5cm]&\vdots &
\phantom{p_1}     & \phantom{p_1} & \phantom{p_1} &
    \phantom{p_1}    &\phantom{p_1} &[0.1cm] \vdots\\
    \phantom{p_1} &\vdots &\phantom{p_1}  &\phantom{p_1}  &\phantom{p_1} & \phantom{p_1} & \phantom{p_1} &\phantom{p_1}  &\vdots & \phantom{p_1} &\phantom{p_1}  &\phantom{p_1}  & \phantom{p_1} & \phantom{p_1} &\phantom{p_1} 
    &\phantom{p_1}  &\phantom{p_1}  &\phantom{p_1}  & \phantom{p_1} & \phantom{p_1} & \phantom{p_1} & \phantom{p_1} & \phantom{p_1} & \phantom{p_1} &\phantom{p_1}  &\phantom{p_1} 
    \phantom{p_1}    &\phantom{p_1} &[0.1cm] \\[2mm]
    \phantom{p_1}    & \phantom{p_1} &\phantom{p_1}  & \phantom{p_1} & \phantom{p_1} & \phantom{p_1} &\phantom{p_1}  &\phantom{p_1}   &\phantom{p_1}  &\phantom{p_1}  &\phantom{p_1}  &\phantom{p_1}  &\phantom{p_1}  &\phantom{p_1}   &\phantom{p_1}  & \phantom{p_1} & \phantom{p_1} & \phantom{p_1} & \phantom{p_1}  &\phantom{p_1}   &\phantom{p_1}   & \phantom{p_1} & \phantom{p_1}  &\phantom{p_1}   &\phantom{p_1}  &\phantom{p_1}  & \phantom{p_1} &[0.1cm] \mathfrak{p}^{(0)}\\
    \phantom{p_1} &\phantom{p_1} &\phantom{p_1} &\phantom{p_1}
    &\phantom{p_1} &\phantom{p_1} &\phantom{p_1} &\phantom{p_1}
    &\phantom{p_1} &\phantom{p_1} &\phantom{p_1} & \phantom{p_1} & \phantom{p_1} &\phantom{p_1}  &\phantom{p_1} 
    & \phantom{p_1} & \phantom{p_1} &\phantom{p_1}  &\phantom{p_1}  & \phantom{p_1} & \phantom{p_1} &\phantom{p_1}  & \phantom{p_1} & \phantom{p_1} &\phantom{p_1}  & \phantom{p_1} &\phantom{p_1}  &\phantom{p_1}  &\phantom{p_1} 
    \\
  };

  \draw[<->] ([yshift=-7.7mm]m-1-2.north west)
  --([yshift=-3mm]m-1-7.south east)
  node[midway,description]{$\alpha_1$}; \draw[<->]
  ([xshift=-5mm]m-2-2.north west)
  -- ([xshift=-5mm]m-8-2.south west) node[midway, description]
  {$\frac{n}{p_1^{\alpha_1}}$};

  \draw[decorate] (m-2-2.north west) -- (m-8-2.south west);
  \draw[decorate] (m-2-2.north east) -- (m-8-2.south east);
  \draw[decorate] (m-2-7.north west) -- (m-8-7.south west);
  \draw[decorate] (m-2-7.north east) -- (m-8-7.south east);

  \draw[dashed] (m-8-7.south east) -- (m-12-5.south west);
  \draw[decorate] (m-8-2.south west) -- (m-12-2.south west);

  \draw[decorate] (m-12-2.south west) -- (m-17-2.north west);
  \draw[decorate] (m-12-3.south west) -- (m-17-3.north west);
  \draw[decorate] (m-12-4.south west) -- (m-17-4.north west);
  \draw[decorate] (m-12-5.south west) -- (m-17-5.north west);

  \draw[decorate] (m-16-2.south west) -- (m-18-2.south west);
  \draw[dashed] (m-16-5.south west) -- (m-18-2.south east);

  \draw[decorate] (m-18-2.south west) -- (m-24-2.south west);
  \draw[decorate] (m-18-2.south east) -- (m-24-2.south east);

  \draw[<->] ([xshift=-7mm]m-19-2.north west) --
  ([xshift=-7mm]m-24-2.south west) node[midway, description]
  {$\frac{n}{p_1}-\frac{n}{p_1^2}$}; \draw[<->]
  ([yshift=-7.7mm]m-11-2.north west) -- ([yshift=-3mm]m-11-4.south
  east) node[midway, description]{$i_1$}; \draw[<->]
  ([xshift=-10mm]m-13-2.north west) -- ([xshift=-10mm]m-16-2.south
  west) node[midway, description]
  {$\frac{n}{p_1^{i_1}}-\frac{n}{p_1^{i_1+1}}$};

  \draw[<->] ([yshift=-7.7mm]m-1-9.north west) --
  ([yshift=-3mm]m-1-13.south east) node[midway,
  description]{$\alpha_2$}; \draw[<->] ([xshift=-5mm]m-2-9.north west)
  -- ([xshift=-5mm]m-6-9.south west) node[midway, description]
  {$\frac{n}{p_2^{\alpha_2}}$};

  \draw[decorate] (m-2-9.north west) -- (m-6-9.south west);
  \draw[decorate] (m-2-9.north east) -- (m-6-9.south east);
  \draw[decorate] (m-2-13.north west) -- (m-6-13.south west);
  \draw[decorate] (m-2-13.north east) -- (m-6-13.south east);

  \draw[decorate] (m-6-9.south west) -- (m-10-9.south west);
  \draw[dashed] (m-6-13.south east) -- (m-10-11.south east);

  \draw[decorate] (m-10-9.south west) -- (m-16-9.north west);
  \draw[decorate] (m-10-10.south west) -- (m-16-10.north west);
  \draw[decorate] (m-10-11.south west) -- (m-16-11.north west);
  \draw[decorate] (m-10-12.south west) -- (m-16-12.north west);

  \draw[decorate] (m-15-9.south west) -- (m-19-9.south west);
  \draw[dashed] (m-15-11.south east) -- (m-19-9.south east);

  \draw[decorate] (m-19-9.south west) -- (m-23-9.south west);
  \draw[decorate] (m-19-9.south east) -- (m-23-9.south east);

  \draw[<->] ([xshift=-7mm]m-20-9.north west) --
  ([xshift=-7mm]m-23-9.south west) node[midway, description]
  {$\frac{n}{p_2}-\frac{n}{p_2^2}$}; \draw[<->]
  ([yshift=-7.7mm]m-9-9.north west) -- ([yshift=-3mm]m-9-11.south
  east) node[midway, description]{$i_2$}; \draw[<->]
  ([xshift=-10mm]m-11-9.north west) -- ([xshift=-10mm]m-15-9.south
  west) node[midway, description]
  {$\frac{n}{p_2^{i_2}}-\frac{n}{p_2^{i_2+1}}$};

\draw[<->] ([yshift=-4mm]m-1-21.south east) --
  ([yshift=-4mm]m-1-27.south east) node[midway,
  description]{$\alpha_r$}; 
  \draw[<->] ([xshift=-5mm]m-2-22.north
  west) -- ([xshift=-5mm]m-4-22.south west) node[midway, description]
  {$\frac{n}{p_r^{\alpha_r}}$};

  \draw[decorate] (m-2-22.north west) -- (m-4-22.south west);
  \draw[decorate] (m-2-22.north east) -- (m-4-22.south east);
  \draw[decorate] (m-2-27.north west) -- (m-4-27.south west);
  \draw[decorate] (m-2-27.north east) -- (m-4-27.south east);

  \draw[decorate] (m-4-22.south west) -- (m-8-22.south west);
  \draw[dashed] (m-4-27.south east) -- (m-8-25.south west);

  \draw[decorate] (m-8-22.south west) -- (m-13-22.north west);
  \draw[decorate] (m-8-23.south west) -- (m-13-23.north west);
  \draw[decorate] (m-8-24.south west) -- (m-13-24.north west);
  \draw[decorate] (m-8-25.south west) -- (m-13-25.north west);

  \draw[decorate] (m-12-22.south west) -- (m-18-22.south west);
  \draw[dashed] (m-12-25.south west) -- (m-18-22.south east);

  \draw[decorate] (m-18-22.south west) -- (m-21-22.south west);
  \draw[decorate] (m-18-22.south east) -- (m-21-22.south east);

  \draw[<->] ([xshift=-7mm]m-19-22.north west) --
  ([xshift=-7mm]m-21-22.south west) node[midway, description]
  {$\frac{n}{p_r}-\frac{n}{p_r^2}$}; \draw[<->]
  ([yshift=-5.7mm]m-7-22.north west) -- ([yshift=-1mm]m-7-24.south
  east) node[midway, description]{$i_r$}; \draw[<->]
  ([xshift=-9.5mm]m-9-22.north west) -- ([xshift=-9.5mm]m-12-22.south
  west) node[midway, description]
  {$\frac{n}{p_r^{i_r}}-\frac{n}{p_r^{i_r+1}}$};

  \draw ([yshift=12mm]m-2-1.north west) -- ([yshift=12mm]m-2-27.north east);
  \draw ([yshift=7mm]m-2-1.north west) -- ([yshift=7mm]m-2-27.north east);

  \draw (m-2-1.north west) -- (m-2-27.north east); \draw (m-2-1.south
  west) -- (m-2-27.south east);

  \draw ([xshift=8mm]m-4-13.south east) -- (m-4-27.south east);
  \draw ([xshift=8mm]m-4-13.north east) -- (m-4-27.north east);
 
  \draw ([xshift=3mm]m-5-7.south east) -- ([xshift=-2mm]m-5-21.south
  east);
  \draw ([xshift=3mm]m-6-7.south east) -- ([xshift=-2mm]m-6-21.south
  east);

  \draw (m-7-1.south west) -- ([xshift=-2mm]m-7-9.south west);
  \draw (m-8-1.south west) -- ([xshift=-2mm]m-8-9.south west);
  \draw (m-7-14.south west) -- ([xshift=-2mm]m-7-21.south east);
  \draw (m-8-14.south west) -- (m-8-27.south east);

  \draw ([xshift=3mm]m-9-12.south east)--(m-9-27.south east);
  \draw (m-10-7.south west) -- ([xshift=-2mm]m-10-19.south east);

  \draw ([xshift=2mm]m-11-6.south west) -- (m-11-27.south east); \draw
  (m-12-1.south west) -- ([xshift=-2mm]m-12-6.south east);
  \draw ([xshift=3mm]m-12-11.south east) -- (m-12-27.south east);

  \draw (m-13-1.south west) -- ([xshift=-2mm]m-13-6.south east); \draw
  ([xshift=3mm]m-13-12.south west) -- ([xshift=-2mm]m-13-21.south
  east);

  \draw (m-14-6.south west) -- ([xshift=-2mm]m-14-21.south east);
  \draw ([xshift=3mm]m-14-24.south east) -- (m-14-27.south east);

  \draw (m-16-1.north west) -- ([xshift=-2mm]m-16-22.north west);
  \draw (m-16-25.north west) -- (m-16-27.north east);

  \draw (m-16-1.south west) -- ([xshift=-2mm]m-16-9.south west); 
\draw (m-16-12.south west) -- ([xshift=-2mm]m-16-22.south west); 
\draw  (m-16-24.south west) -- (m-16-27.south east);

  \draw (m-19-1.north west) -- ([xshift=-3mm]m-19-9.north west); \draw
  (m-19-11.north west) -- (m-19-27.north east);

  \draw (m-20-1.north west) -- (m-20-27.north east); \draw
  ([xshift=2mm]m-21-3.north west) -- (m-21-27.north east);

  \draw ([xshift=2mm]m-22-3.north west) -- ([xshift=-2mm]m-22-6.north
  east); \draw ([xshift=2mm]m-22-10.north west) -- (m-22-27.north
  east);

  \draw ([yshift=3.5mm,xshift=2mm]m-24-3.north west) --
  ([yshift=3.5mm]m-24-13.north east); \draw
  ([yshift=-0.5mm]m-24-1.north west) -- ([yshift=-0.5mm]m-24-5.north
  east); \draw (m-24-1.south west) -- (m-24-5.south east);

  \draw ([xshift=2mm]m-23-3.north west) -- (m-23-13.north east);

  \foreach \i in {26, 27} { \draw (m-\i-1.north west) --
    (m-\i-27.north east) ; };

  \fill[white] (m-1-16.north west) -- (m-1-19.north east) --
  ([xshift=-4mm]m-27-19.south east) -- ([xshift=-4mm]m-27-16.south
  west);

\node (p1) at ([xshift=3mm]m-1-17) {$\dots$};
\node (p2) at ([xshift=3mm]m-4-17) {$\dots$};
\node (p3) at ([xshift=3mm]m-11-17) {$\dots$};
\node (p4) at ([xshift=3mm]m-17-17) {$\dots$};
\node (p5) at ([xshift=3mm]m-22-17) {$\dots$};
\end{tikzpicture}
 \caption{Algorithm seen through Young Diagrams}
 \label{fig:algo-as-yd}
 \end{sidewaysfigure}
\pagebreak
\begin{proof}\mbox{}
  \begin{enumerate}
    \item The multiplicity of $1$ in $E(A_n)$ is 
      \begin{align*}
        &n - \max\{\text{height of } p_i \text{ tableau} \mid i=1, \dots, r\} \\
        &= n - \frac{n}{p_1}
      \end{align*}
    \item Let $d > 1$ be the least elementary divisor of $A_n$. Then $d$ is the product of $p_i$s in the lowest non-empty row. Since the $p_1$-tableau is the tallest having one box in the last row, this $d$ must be $p_1$. The multiplicity of $p_1$ is
      \begin{align*}
        &\frac  n {p_1} - \text{Index of the row containing the second last corner} \\
        &= \begin{cases}
            \frac n {p_1} - \frac n {p_2},   &\text{ if } \frac n {p_2} > \frac n {p_1^2}\\
            \frac n {p_1} - \frac n {p_1^2}, &\text{ if } \frac n {p_2} < \frac n {p_1^2}  
           \end{cases}
      \end{align*}
      which proves the claim.
    \item The largest elementary divisor is the product of the numbers in the first row of Figure~\ref{fig:algo-as-yd}. This number is clearly $n$. The multiplicity of $n$ in $E(A_n)$ is the index of the row containing the first corner. We see that this multiplicity is
      \begin{align*}
        &\quad\min\left\{\frac{n}{p_i^{\alpha_i}}\bigg| 1 \leqslant i \leqslant r\right\} \\
        &=\frac{n}{\max\left\{p_i^{\alpha_i}: 1 \leqslant i  \leqslant r\right\}}
      \end{align*}
This completes the proof. \qedhere
  \end{enumerate}
\end{proof}

This proposition and its proof suggest that
Figure~\ref{fig:algo-as-yd}, in principle, gives a ``formula'' for
the elementary divisors and their muliplicities, equivalently, the
multiset $E(A_n)$.
\begin{thm}\label{pro:characterising-ele-divisor-ratios}
  \mbox{}
  \begin{enumerate}[ref=(\arabic*)]
  \item The $n$-tuple $\left(e_1, \frac{e_2}{e_1}, \dots, \frac{e_n}{e_{n-1}}\right)$ is a permutation of the $n$-tuple
    \[(\underbrace{p_1, \dots, p_1}_{\alpha_1 \text{ times}}, \underbrace{p_2, \dots, p_2}_{\alpha_2 \text{ times}},\dots,\underbrace{p_r,\dots,p_r}_{\alpha_r \text{times}},\underbrace{1,\dots,1}_{n-\sum_i\alpha_i \text{ times}})\]
  \item \label{pro-part:index-of-occurence} The ratio $\frac{e_{j}}{e_{j-1}}$ is $p_i$ if and only if $j =
n-\frac{n}{p_i^t}+1$ for some $t$ satisfying $1 \leqslant t \leqslant
\alpha_i$. 
  \end{enumerate}
\end{thm}
\begin{proof}
  \begin{enumerate}
  \item We need to prove that $\frac{e_i}{e_{i-1}}$ is a prime divisor of $n$. Note that this claim is equivalent to proving that every row in Figure~\ref{fig:algo-as-yd} has atmost one corner. That is, exactly one of the $p_i$ tableau has a corner. Towards a contradiction, assume that there are two distinct primes $p_{i_1}$ and $p_{i_2}$ whose tableaux for $n$ have a corner each in the same row $R$:
    
    Thus, there are indices $l_1$ and $l_2$ with $0 \leqslant l_1 \leqslant \alpha_{i_1}-1$ and $0 \leqslant l_2 \leqslant \alpha_{i_2}-1$  such that 
    \begin{align*}
      \label{eq:row-index-for-corner}
      R &= \sum_{j=0}^{l_1} \frac{n}{p_{i_1}^{\alpha_{i_1}}}\phi(p_{i_1}^{j})\\
        &= \sum_{j=0}^{l_2} \frac{n}{p_{i_2}^{\alpha_{i_2}}}\phi(p_{i_2}^{j})
    \end{align*}
This equality implies that \[\frac{n}{p_{i_1}^{\alpha_{i_1}-l_1}} = \frac{n}{p_{i_2}^{\alpha_{i_2}-l_2}}\]
which is a contradiction, since $p_{i_1}$ and $p_{i_2}$ are distinct primes.
\item Notice that index $k$ of the $n$-tuple contains $p_i$ if the row $k$ contains a corner of the $p_i$-tableau. These indices are therefore given by 
\[\left\{\frac{n}{p_i^{\alpha_i}}\sum_{j=0}^r\phi(p_i^{\alpha_i-j})+1 \bigg| r=0,1, \dots,\alpha_i-1\right\}.\]
Thus, we get, the indices which contain $p_i$ are
\begin{align*}
  \left\{n-\frac{n}{p_i^{r}}+1 \bigg| r=1, \dots,\alpha_i\right\}
\end{align*}
which completes the proof. \qedhere
  \end{enumerate}
\end{proof}
 We leave it to the reader to find another proof of
Proposition~\ref{pro:some-ele-div-stat} using Theorem~\ref{pro:characterising-ele-divisor-ratios}.

In the following sections, we shall describe a basis for $G(n)$ in terms of the standard basis of $\oplus_{d|n} \ZZ[x]/\Phi_d(x)$ through a
pictorial algorithm: by a basis for $G(n)$ is meant a set of
generators
\[\{0 \oplus 0\oplus\cdots\oplus\underbrace{ 1 }_{i\text{th place}}\oplus 0 \oplus \cdots\oplus 0: 0 \leqslant i \leqslant n-1\}\]
 for the abelian group $\bigoplus \ZZ/e_n(i)\ZZ$ where $(e_n(0),
\dots, e_n(n-1))$ is the tuple of elementary divisors for $A_n$.
\section{Setup for the Algorithm}
\label{sec:basis-for-G1n}
In this section, we will state the definitions and prove some basic
lemmas that are instrumental to the algorithm in the next section.

To determine a basis for $G(n)$, it suffices to find a basis
$\{\mathfrak{p}^{(j)}:0 \leqslant j \leqslant n-1\}$ for $\bigoplus_{d \mid n}
  \ZZ[X]/\Phi_d(X)$ so that 
\begin{equation}
  \label{eq:generators-for-G1n}
  \{\overline{\Psi}_n(\mathfrak{p}^{(j)}): e_n(j) > 1\}
\end{equation}
is a set of generators for the abelian group $G(n)$ with respect to
which the relations are the simplest possible:
\begin{equation}
  \label{eq:relations-for-G1n}
  e_n(j) \overline{\Psi}_n(\mathfrak{p}^{(j)}) = 0.
\end{equation}
This idea is captured by the following definition:
\begin{defn}
\label{defn:smith-vec}
Given a positive integer $n$, let $(e_n(0), \dots, e_n(n-1))$ be the tuple
  of elementary divisors of $A_n$. We say that $(\mathfrak{p}^{(j)}: 0 \leqslant
  j \leqslant n-1)$ is a {\bfseries Smith vector} for $n$ if: 
\begin{enumerate}
\item $\left\{\mathfrak{p}^{(j)} : 0 \leqslant j
    \leqslant n-1 \right\}$ is a $\ZZ$-basis of $\bigoplus_{d \mid n}
  \ZZ[X]/\Phi_d(X)$ and
\item $a_j\mathfrak{p}^{(j)} \in \Im(\Psi_n)$
  if and only if $e_n(j) \mid a_j$.
\end{enumerate} 
\end{defn}
The following lemma will tell us how to compute Smith vector for $n$: 
\begin{lem}
  \label{lem:smith-vectors-are-columns-of-AV}
  If $(\mathfrak{p}^{(j)})$ is a Smith vector for $n$, then there exists $U_n,
  V_n \in GL_n(\mathbf{Z})$ such that $S(A_n) = U_nA_nV_n$ and
  $A_nV_n(\overline{X}^{j}) = e_n(j)\mathfrak{p}^{(j)}$ where \[(e_n(0), \dots, e_n(n-1))\]
 is the tuple of elementary divisors of $A_n$. 
\end{lem}
\begin{proof}
  We introduce a notation for the standard basis of the direct sum
  $\oplus_{d \mid n} \ZZ[X]/\langle\Phi_d(X)\rangle$; for a divisor
  $d$ of $n$, and for every $i$ such that $0 \leqslant i \leqslant
  \phi(d)-1$, put:
  \begin{align*}
    g_{i,d}(X) := 0 \oplus \dots \oplus 0 \oplus X^i
    \bmod{(\Phi_d(X))} \oplus 0\oplus \dots\oplus 0. 
  \end{align*}
  Let $U_n$ be the endomorphism of $\oplus_{d \mid n}
  \ZZ[X]/\langle\Phi_d(X)\rangle$  that exchanges the basis underlying the given Smith vector with the standard basis:
  \[\mathfrak{p}^{(j)} \mapsto g_{i,d} \text{ where } j = \sum_{d' <d , d'\mid
    n} \phi(d') + i.\]
  Thus, $U_n$ is invertible. Since $e_n(j)\mathfrak{p}^{(j)}$ is in the image of $A_n$, it follows that there are vectors $h^{(j)} \in \ZZ[X]/\langle X^n-1\rangle$ such that 
  \[A_n(h^{(j)}) = e_n(j)\mathfrak{p}^{(j)}.\]
  Define the map $V_n:\ZZ[X]/\langle X^n-1\rangle \to \ZZ[X]/\langle X^n-1\rangle$ as follows:
  \[\overline{X}^{j} \mapsto h^{(j)},  \quad 0 \leqslant j \leqslant n-1\]
  Clearly, $U_n A_n V_n = S(A_n)$. It suffices to check that $V_n$ is
  an isomorphism, that is, $\det(V_n) = \pm 1$: 
  \begin{align*}
    U_n A_n V_n = S(A_n) &\Rightarrow \det(U_n) \det(A_n) \det(V_n) = \det(S(A_n)) \\
                        &\Rightarrow \det(A_n) \det(V_n) =
                        \pm|\det(A_n)| \text{ (since } \det(U_n) =
                        \pm 1)\\
                        &\Rightarrow \det(V_n) = \pm 1 \text{ (since }
                        \det(A_n) \neq 0).
  \end{align*}
This completes the proof.
\end{proof}
From Lemma~\ref{lem:smith-vectors-are-columns-of-AV}, we see that 
$\{e_n(j) \mathfrak{p}^{(j)}: 0 \leqslant j \leqslant n-1\}$
is a basis for the image of $\Psi_n$ and we have the following
isomorphism of $\ZZ$-modules:
\begin{equation}
  \label{eq:iso-class-of-G1n}
  G(n) \simeq \bigoplus_{e_n(j) > 1} \langle
  \overline{\Psi}_n(\mathfrak{p}^{(j)})\rangle \simeq \bigoplus_{e_n(j) > 1} \ZZ/e_n(j) \ZZ.
\end{equation} 

In Lemma~\ref{thm:smith-and-coprime}, for relatively prime positive
integers $m$ and $n$, we have shown that $S(A_{mn}) = S(A_m \otimes A_n)$. It is now natural to ask if Smith vectors for
$m$ and $n$ can be coaxed to produce a Smith vector for $mn$. In the commutative diagram of maps
in Figure~\ref{fig:tensor-prod-of-smith}, since both the rows are exact
and $P_{m,n}$ and $T(m,n)$ are isomorphisms, a straightforward diagram
chasing proves that $f_{m,n}$ is an isomorphism (see also \cite[Lemma~7.1]{LangAlg2002}). 
\begin{figure}[htbp]
\[
\begin{tikzcd}[xscale=0.4]
0 \ar{r} & \frac{\ZZ[X]}{\langle X^m-1 \rangle}\otimes
\frac{\ZZ[Y]}{\langle Y^n-1 \rangle} \ar{r}{\Psi_m
  \otimes \Psi_n}\ar{d}{P_{m,n}}&
{\bigoplus_{\substack{d_1 \mid m \\ d_2 \mid
    n}}}\frac{\ZZ[X]}{\langle \Phi_{d_1}(X) \rangle}\otimes
\frac{\ZZ[Y]}{\langle \Phi_{d_2}(Y)
  \rangle}\ar{r}{\overline{\Psi_m\otimes \Psi_n}}\ar{d}{T(m,n)}& G(m,n) \ar{r}\ar{d}{f_{m,n}}&0\\
0 \ar{r}& \frac{\ZZ[t]}{\langle t^{mn}-1 \rangle} \ar{r}{\Psi_{mn}}& {\bigoplus_{d \mid mn}} \frac{\ZZ[t]}{\langle\Phi_d(t)\rangle} \ar{r}{\overline{\Psi_{mn}}}& G(mn) \ar{r}& 0 
\end{tikzcd}
\]
\caption{Tensor Product of Smith Vectors}
\label{fig:tensor-prod-of-smith}
\end{figure}
In the next lemma, we consider the tensor product of Smith vectors in
the top row of Figure~\ref{fig:tensor-prod-of-smith} and study its
properties in the bottom row.
\begin{lem}\mbox{}
\label{lem:tensor-prod-of-smith-vectors-works}
 Let $m$ and $n$ be relatively prime positive integers. Let
 $\left\{(\mathfrak{p}^{(j)}): 0 \leqslant j \leqslant m-1\right\}$ and
 $\left\{(\mathfrak{q}^{(j)}) : 0 \leqslant j \leqslant n-1\right\}$ be Smith
 vectors for $m$ and $n$ respectively. Then, in the group $G(mn)$, 
\begin{enumerate}
\item \label{lem-pt-1:order}the order of the element 
$\overline{\Psi}_{mn}(\mathfrak{p}^{(i)}(\overline{t}^n)\mathfrak{q}^{(j)}(\overline{t}^m))$ is $e_m(i)e_n(j)$,
where $e_k(\iota)$ denotes the $\iota$th elementary divisor of $A_k$.
\item  \label{lem-pt-2:intersection}suppose that
  $\overline{\Psi}_{mn}(\mathfrak{p}^{(i_1)}(\overline{t}^n)
  \mathfrak{q}^{(j_1)}(\overline{t}^m)) $ and
  $\overline{\Psi}_{mn}(\mathfrak{p}^{(i_2)}(\overline{t}^n)\mathfrak{q}^{(j_2)}(\overline{t}^m))
  $ are non-zero. Then, the intersection 
\[\langle \overline{\Psi}_{mn}(\mathfrak{p}^{(i_1)}(\overline{t}^n) \mathfrak{q}^{(j_1)}(\overline{t}^m))\rangle \cap \langle \overline{\Psi}_{mn}(\mathfrak{p}^{(i_2)}(\overline{t}^n)\mathfrak{q}^{(j_2)}(\overline{t}^m))\rangle\]
of subgroups is non-trivial if and only if $i_1 = i_2$ and $j_1 = j_2$.  
\end{enumerate}
\end{lem}
In this lemma, we interpret $\mathfrak{p}^{(j)}(\overline{t}^n)$ as the element
  $\bigoplus_{d \mid n}\mathfrak{p}_d^{(j)}(\overline{t}^n)$ which belongs to the
  direct sum $\bigoplus_{d \mid n}\ZZ[t]/\Phi_d(t)$. And, the
  product $\mathfrak{p}^{(i)}\mathfrak{q}^{(j)}$ is to be interpreted as the element
  $\bigoplus_{\substack{d_1 \mid m \\ d_2 \mid n}} \mathfrak{p}_{d_1}^{(i)}
  \mathfrak{q}_{d_2}^{(j)}$ which belongs to $\bigoplus_{\substack{d_1 \mid m \\
      d_2 \mid n}} \ZZ[t]/\Phi_{d_1d_2}(t)$. 
\begin{proof}
The set $\left\{\mathfrak{p}^{(i)}\otimes \mathfrak{q}^{(j)}: \substack{0
  \leqslant i \leqslant m-1\\ 0 \leqslant j \leqslant n-1}\right\}$ is
 a basis for $\bigoplus_{\substack{d_1 \mid m\\d_2 \mid
     n}}\frac{\ZZ[X]}{\Phi_{d_1}(X)} \otimes
 \frac{\ZZ[Y]}{\Phi_{d_2}(Y)}$ since $\mathfrak{p}^{(i)}$ and
 $\mathfrak{q}^{(j)}$ are Smith vectors for $m$ and $n$
 respectively. Since $T(m, n)$ is an isomorphism, we get that 
 the set $\left\{\mathfrak{p}^{(i)}(\overline{t}^n)\mathfrak{q}^{(j)}(\overline{t}^m): \substack{0
   \leqslant i \leqslant m-1\\ 0 \leqslant j \leqslant n-1}\right\}$ is
 a $\ZZ$-basis for the codomain of $\Psi_{mn}$. From the linear
independence of these vectors, \eqref{lem-pt-2:intersection}
 follows. 

By Lemma~\ref{lem:smith-vectors-are-columns-of-AV}, there are
isomorphisms $V_m$ and $V_n$ such that:
\begin{equation}
  \label{eq:pf-of-lemma-smith-works}
  A_mV_m(\overline{X}^i) = e_m(i) \mathfrak{p}^{(i)} \text{ and }  A_nV_n(\overline{Y}^j) = e_n(j) \mathfrak{q}^{(j)}
\end{equation}
for $0 \leqslant i \leqslant m-1$ and $0 \leqslant j \leqslant
n-1$. Since $\overline{X}^i \otimes \overline{Y}^j$ is a basis for
$\frac{\ZZ[X]}{(X^m-1)} \otimes \frac{\ZZ[Y]}{(Y^n-1)}$, we note that
$\mathfrak{p}^{(i)} \otimes \mathfrak{q}^{(j)}$ is a basis for the
image of $\Psi_m \otimes \Psi_n$. Now, using the fact that $T(m,n)$
and $P_{m,n}$ are isomorphisms and the commutativity of
Figure~\ref{fig:tensor-prod-of-smith}, we have the set $
 \{e_m(i)e_n(j)\mathfrak{p}^{(i)}(\overline{t}^n)\mathfrak{q}^{(j)}(\overline{t}^m): \substack{0
   \leqslant i \leqslant m-1\\ 0 \leqslant j \leqslant n-1}\}$ is a basis for
the image of $\Psi_{mn}$ from which \eqref{lem-pt-1:order} follows.
\end{proof}
\begin{rem}
\label{rem:general-k-1-k-2}
More generally, we note that if $k_1$ and $k_2$ are two relatively
prime positive integers, then, we may obtain the young diagram for
$k_1k_2$ by repeating $k_2$ times, the rows of the Young diagram for
$k_1$ and $k_1$ times, the rows of the Young diagram for $k_2$. To see
this, suppose that 
\begin{equation*}
k_1 = p_1^{\alpha_1} \dots p_r^{\alpha_r} \text{ and }
k_2= q_1^{\beta_1} \dots q_s^{\beta_s}
\end{equation*}
where $\alpha_i, \beta_j > 0$ and $\{p_i: 1 \leqslant i \leqslant r\}$
and $\{q_j: 1 \leqslant j \leqslant s\}$ are
disjoint. Then, in the $p_j$ tableau for $k_1$, the row with $e$ boxes (where $0
\leqslant e \leqslant \alpha_j$) appears
$\phi(p_j^{\alpha_j-e})k_1p_j^{-\alpha_j}$ times, while this
row appears $\phi(p_j^{\alpha_j-e})k_1k_2p_j^{-\alpha_j}$ in
the $p_j$ tableau for $k_1k_2$. This proves our claim. 
\end{rem}
From the lemma, we infer that tensoring Smith
vectors does not work, because $\left\{e_m(i)e_n(j): \substack{0
  \leqslant i \leqslant m-1\\ 0 \leqslant j \leqslant n-1}\right\}$ is
not the set of elementary divisors of $A_{mn}$. However, we  will see how to
get over this difficulty in the next section. We conclude this section
with the following lemma which we will need in the next section and is
interesting in its own right. 
\begin{lem}
  \label{lem:diagonals-condn-can-be-met}
  Let $P$ be a permutation matrix and $m$ and $n$ be relatively prime
  positive integers. Then, there are diagonal matrices $D_1$ and $D_2$
  satisfying:
\begin{align*}
\gcd(\det(D_1), m) &= 1,\\
\gcd(\det(D_2), n) &= 1 \text{ and} \\
\det(nD_1+mD_2P) &= 1.
\end{align*}
\end{lem}
\begin{proof}
  Let $P$ be a $k \times k$ permutation matrix whose associated
  permutation is $\pi$. We shall present an algorithm to find integers $\{a_i\}_{i=1}^k$ and $\{b_j\}_{j=1}^k$ such that  
\[D_1 = \diag(a_1, \dots, a_k) \quad \text{and} \quad  
D_2 = \diag(b_1, \dots, b_k)\]
satisfy the requirements of lemma.
 
{\bfseries Step 1.} If $\pi$ is a $k$-cycle, then, one may compute $\det(nD_1+mD_2P)$
by the usual formula: 
\[\det(nD_1+mD_2P) = \sum_{\sigma \in S_k} (-1)^{\text{sgn}(\sigma)}
x_{1\sigma(1)}\dots x_{k\sigma(k)}\]
where $x_{ij}$ are the entries of the matrix $nD_1 + mD_2P$. 

For a permutation $\sigma$ 
\begin{equation}
\label{eq:contrib-perm}
x_{1\sigma(1)} \dots x_{k\sigma(k)}\neq 0 \Rightarrow \;\forall\; 1
\leqslant i \leqslant k \quad (\sigma(i) = i \text{ or } \sigma(i) =
\pi(i)).
\end{equation} 
Let $\sigma$ be not the identity permutation such that $x_{1\sigma(1)} \dots x_{k\sigma(k)}\neq 0$. By \eqref{eq:contrib-perm} there exists $j$
such that $\sigma(j) = \pi(j) \neq j$. Set $S = \{\sigma^\ell(j): 1
\leqslant \ell \leqslant k\}$. No element of $S$ is fixed
by $\sigma$, because if $\sigma(\sigma^\ell(j)) = \sigma^\ell(j)$,
then $\sigma(j) = j$, a contradiction. Thus, $\left.\sigma\right|_S =
\left.\pi\right|_S$. But, $\{\pi^\ell(j): 1 \leqslant \ell \leqslant
k\} = \{1, \dots, k\}$ since $\pi$ is a $k$-cycle. Since
$\sigma^\ell(j) = \pi^\ell(j)$, the set $S$ is all of
$\{1, \dots, k\}$. Thus, $\sigma$ must be $\pi$. Hence:
\[\det(nD_1+mD_2P) = n^k \prod_{l=1}^k a_l + (-1)^{k-1} m^k
\prod_{l=1}^k b_l.\] Since $\gcd(m^k, (-1)^{k-1}n^k) =1$, there exists $u$ and $v$ such that 
\begin{equation*} 
m^ku + (-1)^{k-1}n^k v = 1.
\end{equation*}
Then, it is easy to verify that the following choices 
  \begin{align*}
    a_{1} & = u\\
    a_{l} & = 1 \text{ for all } l \neq 1\\
    b_{1} & = v\\
    b_{l} & = 1\text{ for all } l \neq 1
  \end{align*}
  meet the requirements of the lemma.

{\bfseries Step 2.} If $\pi$ is not a cycle, let the cycle
decomposition of $\pi$ be $\pi_1\dots \pi_r$. We may now repeat Step 1
on each of the cycles $\pi_i$ and determine the scalars $a_j$ and
$b_j$ for those $j$  not fixed by $\pi_i$.   
\end{proof}
For an alternative proof of this lemma, see \cite{mo-permute}.
\section{An Algorithm for determining the Smith Vector for $n$}
\label{sec:algo-for-smith-vectors}
Given a positive integer $n$ and its
prime factorisation
\[n = p_1^{\alpha_1}\dots p_r^{\alpha_r},\]
we derived a formula for the elementary divisors of $\Psi_n$ using $r$ Young
tableaux, one for each prime $p_i$.  Now, we will use the same set of
tableaux diagrams to give an algorithm to find a Smith vector for $n$.  
The algorithm will use the following
three modules:  
\begin{description}
\item[$\boldsymbol{SV(p^e)}$] Construct a Smith vector for $p^e$, a prime power.
\item[$\boldsymbol{TSV(k_1, k_2, \mathfrak{p}, \mathfrak{q})}$] Given
  Smith vectors $\mathfrak{p}$ and $\mathfrak{q}$, respectively, for relatively prime
  positive integers $k_1$ and $k_2$, construct a Smith vector for $k_1k_2$.
\item[$\boldsymbol{SV(n)}$] Construct a Smith vectors for $n$ inductively on prime powers. 
\end{description}
 We will illustrate these modules with the example $n = 6$. We will
 calculate the bit complexity of each module in
 Section~\ref{ssec:time-complexity}. 

These algorithms are implemented in SAGE
 (\url{http://www.sagemath.org}) and the code is available from this
 Git (\url{http://github.com}) repository 
\begin{center}
\url{https://github.com/kamalakshya/Cyclotomy/} 
\end{center}
in the file \texttt{final\_smith\_form.sage}. We also calculated the Smith vector for $n$ between $2$ and $30$: this
data is available in the same repository in the file
\texttt{two\_thirty.pdf}.

\subsection{$\boldsymbol{SV(p^e)}$}
Recall from Lemma~\ref{lem:smith-vectors-are-columns-of-AV} that upto
scaling by elementary divisors, the columns of the matrix $W_n =
A_nV_n$ give a Smith vector for $n$. We have
established that $W_{p^e}$ has a simple recursive form (see Remark~\ref{rem:totality-of-col-W}). 
Therefore, a Smith vector $\operatorname{SV}(p^e)$ for $p^e$ can be
computed from this recursive formula. 
\begin{rem}
\label{rem:ht-of-prime-power-smith}
Observe that non-zero coefficients in the Smith
vector $\operatorname{SV}(p)$ are $\pm 1$. Since the non-zero entries of the
matrix $A_{p^e}$  are $\pm 1$, it follows by induction on $e$ that the
non-zero coefficients in $\operatorname{SV}(p^e)$ are $\pm 1$. We will
use this fact on the calculation of bit complexity of the algorithm. 
\end{rem}

{\bfseries Examples.} 
\begin{align}
  \operatorname{SV}(2) &= (\overline{1}\oplus 0, \overline{1} \oplus \overline{1})\label{eq:smith-vector-for-2}\\
  \operatorname{SV}(3) &= (\overline{1} \oplus 0, \overline{1} \oplus \overline{t}, \overline{1} \oplus \overline{1})\label{eq:smith-vector-for-3}
\end{align}

\subsection{$\boldsymbol{TSV(k_1, k_2, \mathfrak{p}, \mathfrak{q})}$}
\label{ssec:meat-of-algo}
 This module is the most crucial part of our algorithm. Let us first
 do a pictorial construction and make some observations about it. 

{\bfseries Construction.} To construct the tableaux diagram for $k_1k_2$
from the tableaux diagram for $k_1$ and that of $k_2$, we need
$k_2$ repetitions of the rows of the tableaux diagram for $k_1$ and
$k_1$ repetitions of the rows of the tableaux diagram for $k_2$ (See Remark~\ref{rem:general-k-1-k-2}).
Throughout this subsection, it will be convenient to assume that the
rows of the tableaux diagram for $n$ are numbered from bottom to top with
indices between $0$ and $n-1$. We will also index repetition of a row
from the bottom with indices  between $0$ and $k_i - 1$. So, the components of the Smith vector attached to the rows of the
tableaux diagram for $k_1$ and $k_2$ must change as follows:  
\begin{itemize} 
\item to the $j_1$th repetition of $i_1$th row of tableaux diagram of
  $k_1$, we attach the vector $k_2 \mathfrak{p}^{(i_1)}(\overline{t}^{k_2})
  \mathfrak{q}^{(j_1)} (\overline{t}^{k_1})$,
\item to the $i_2$th repetition of $j_2$th row of the tableaux diagram
  of $k_2$, we attach the vector $k_1 \mathfrak{p}^{(i_2)}(\overline{t}^{k_2}) \mathfrak{q}^{(j_2)} (\overline{t}^{k_1})$.
\end{itemize}
Here, we have multiplied by $k_i$ so that the image of the
corresponding element under $\overline{\Psi}_{mn}$ has order dividing
$k_{3-i}$. Finally, we will juxtapose these tableaux diagrams for
$k_1k_2$ so that the row indices match. This construction is
demonstrated in Figure~\ref{fig:gen-pic-const}.
\begin{sidewaysfigure}
\vspace*{12cm}
\hspace*{-1cm}
\begin{tikzpicture}[
decoration={markings,
    mark connection node=my node,
    mark=at position .5 with {\node[transform shape] (my node) {$...$};}
    }, every node/.style={font={\footnotesize},outer sep=0pt,inner sep=2pt}
]
\matrix[ampersand replacement = \&, matrix of math nodes, execute at empty cell = \node{\phantom{p_1}};] (m)
{
\& \phantom{p_1}                             \& \phantom{p_1}\&\phantom{p_1} \&  \&     \& \& \&k_1 \& \&  \& \& \& \&  \& \& \phantom{p_1}\&\phantom{p_1} \& \& \&   \&  \& k_2\& \& \& \\ 
\&\text{Row Index}   \& \phantom{p_1}\& \phantom{p_1}\& \phantom{p_1} \&     \&       \& \& \&      \&  \& \&        \&  \& \& \& \&      \&  \& \& \& \& \\       
\&     k_1k_2-1         \& \phantom{p_1}\& \phantom{p_1}\&\phantom{p_1}  \&  \phantom{p_1}    \&\phantom{p_1} \&\phantom{p_1} \&\phantom{p_1} \&   \phantom{p_1}   \&  \phantom{p_1}\& \phantom{p_1}\&\phantom{p_1}
\& \phantom{p_1} \&\phantom{p_1}\&k_2\mathfrak{p}^{(k_1-1)}(\overline{t}^{k_2})\mathfrak{q}^{(k_2-1)}(\overline{t}^{k_1})
\&\phantom{p_1} \& \phantom{p_1}\& \phantom{p_1}     \& \phantom{p_1} \& \phantom{p_1}\& \phantom{p_1}\& \phantom{p_1}\& \phantom{p_1}\& \phantom{p_1}\&\phantom{p_1} \& \phantom{p_1}\&\phantom{p_1}\&\phantom{p_1}\&\phantom{p_1} \& k_1\mathfrak{p}^{(k_1-1)}(\overline{t}^{k_2})\mathfrak{q}^{(k_2-1)}(\overline{t}^{k_1})\\   
\&                             \& \& \&  \&     \& \& \& \&      \&  \& \& \&  \&     \&  \& \& \& \&      \&  \& \& \& \& \& \\ 
\&  \vdots                \& \& \&  \&     \& \& \& \&      \&  \& \&
\&  \&     \&  \vdots\& \& \& \&      \&  \& \& \& \& \& \& \& \& \&
\& \vdots\\ 
\&                             \& \&\phantom{p_1} \&  \&      \& \& \& \&      \&  \& \& \&  \&     \&  \& \& \& \phantom{p_1}\&\phantom{p_1}      \& \phantom{p_1} \& \phantom{p_1}\& \& \& \&\& \phantom{p_1}\& \phantom{p_1}\&\\ 
\& \ell                            \& \& \&  \&     \& \& \& \&      \& \phantom{p_1} \& \phantom{p_1}\&\phantom{p_1} \&\phantom{p_1}
\&\phantom{p_1}\&  k_2\mathfrak{p}^{(i_1)}(\overline{t}^{k_2})\mathfrak{q}^{(j_1)}(\overline{t}^{k_1})   \&  \& \& \& \&      \&  \& \& \& \& \&\& \& \& \&k_1\mathfrak{p}^{(i_2)}(\overline{t}^{k_2})\mathfrak{q}^{(j_2)}(\overline{t}^{k_1})\\ 
\&                             \& \& \&  \&      \& \& \& \&      \&  \& \& \&  \&     \&  \& \& \& \&      \&  \& \& \& \& \&\&\\ 
\&                             \& \phantom{p_1}\& \&  \&      \& \& \&\cdots \&     \& \&  \&     \&  \& \& \& \&      \&  \&  \& \&\&\cdots \&\\ 
\&    \vdots             \& \phantom{p_1}\&\phantom{p_1} \&  \phantom{p_1}\&    \phantom{p_1}  \& \phantom{p_1}\&\phantom{p_1} \& \phantom{p_1}\&      \&  \phantom{p_1}\& \&\phantom{p_1}  \&     \&  \& \& \& \&  \phantom{p_1}    \&  \phantom{p_1}\& \phantom{p_1}\& \phantom{p_1}\&\phantom{p_1} \&\phantom{p_1}
\phantom{p_1}\&\phantom{p_1}\& \& \& \&\\      
 \&                             \& \& \&  \&       \& \phantom{p_1}\& \&\phantom{p_1} \&      \&  \& \phantom{p_1}\&\phantom{p_1} \&  \phantom{p_1}\&    \phantom{p_1} \&  \& \& \& \&      \&  \& \& \& \& \& \\ 
\&                             \& \phantom{p_1}\& \phantom{p_1}\& \phantom{p_1} \&  \phantom{p_1}    \& \phantom{p_1}\& \phantom{p_1}\& \phantom{p_1}\&  \phantom{p_1}    \&  \phantom{p_1}\& \& \&  \&     \&  \& \& \& \&      \&  \& \& \& \& \&\\ 
\&                             \& \& \&  \&      \& \& \& \&     \phantom{p_1} \&  \phantom{p_1}\&\phantom{p_1} \& \&  \&     \&  \phantom{p_1}\& \phantom{p_1}\& \phantom{p_1}\& \phantom{p_1}\&      \&  \& \& \& \& \&\\ 
\&                             \& \& \&  \&      \& \& \& \&      \&
\& \& \&  \&     \&\vdots  \& \& \& \&      \&  \& \& \& \& \& \& \&
\& \& \& \vdots\\
\&\frac{k_1k_2}{p_1}\& \phantom{p_1}\& \phantom{p_1}\& \phantom{p_1} \&      \& \& \& \&      \&  \& \&        \&  \& \& \& \&      \&  \& \& \& \\             
\&                             \& \& \&  \&      \& \& \& \&      \&  \& \& \&  \&     \&  \& \& \& \&      \&  \& \& \& \& \&\\
\&                             \& \& \&  \&      \& \& \& \&      \&  \& \& \&  \&     \&  \& \& \& \&      \&  \& \& \& \& \&\\
\&     \vdots             \& \phantom{p_1}\&\phantom{p_1} \&  \phantom{p_1}\&      \& \& \& \&      \&  \& \& \&  \&     \&  \& \& \& \&      \&  \& \& \& \& \& \\
\&                             \& \& \&  \&      \& \& \& \&      \&  \& \& \&  \&     \&  \& \& \& \&      \&  \& \& \& \& \& \\
\&           1                \& \& \&  \&     \& \& \& \&      \&  \& \& \&  \&\&k_2\mathfrak{p}^{(0)}(\overline{t}^{k_2})\mathfrak{q}^{(1)}(\overline{t}^{k_1})     \&  \& \& \& \&      \&  \& \&  \& \& \& \& \& \& \&k_1\mathfrak{p}^{(1)}(\overline{t}^{k_2})\mathfrak{q}^{(0)}(\overline{t}^{k_1})\\
\&           0                \& \phantom{p_1}\&\phantom{p_1} \&  \phantom{p_1}\&      \& \& \& \&      \&  \& \& \&  \&\&k_2\mathfrak{p}^{(0)}(\overline{t}^{k_2})\mathfrak{q}^{(0)}(\overline{t}^{k_1})
\&  \& \& \& \&      \&  \& \& \& \& \& \& \& \& \&
k_1\mathfrak{p}^{(0)}(\overline{t}^{k_2})\mathfrak{q}^{(0)}(\overline{t}^{k_1})\\
\&                             \& \& \&  \&      \& \& \& \&      \&  \& \& \&  \&     \&  \& \& \& \&      \&  \& \& \& \& \&\\
\&                             \& \& \&  \&      \& \& \& \&      \&
\& \& \&  \& \phantom{p_1}    \& \phantom{p_1} \&\phantom{p_1} \& \& \&      \&  \& \& \& \& \&\\ 
\&                             \& \& \&  \&      \& \& \& \&      \&  \& \& \&  \&     \&  \phantom{p_1}\&\phantom{p_1} \& \phantom{p_1}\& \phantom{p_1}\&      \&  \& \& \& \& \&\\
\&                             \& \& \&  \&      \& \& \& \&      \&  \& \& \&  \&     \&  \& \& \& \&      \&  \& \& \& \& \&\\
};
\draw[decorate] (m-3-4.north west) -- (m-15-4.south west);
\draw (m-15-4.south west) -- (m-15-4.south east) -- (m-15-4.north east);
\draw[dashed] (m-15-4.north east) -- (m-10-6.north east);
\draw (m-3-4.north west) -- (m-3-7.north east) -- (m-10-7.north east) -- (m-10-7.north west);

\draw[decorate] (m-3-11.north west) -- (m-12-11.south west);
\draw (m-12-11.south west) -- (m-12-11.south east) -- (m-12-11.north east);
\draw[dashed] (m-12-11.north east) -- (m-7-13.north east);
\draw (m-3-11.north west) -- (m-3-14.north east) -- (m-7-14.north east) -- (m-7-14.north west);

 \draw[decorate] (m-3-18.north west) -- (m-13-18.south west);
\draw (m-13-18.south west) -- (m-13-18.south east) -- (m-13-18.north east);
\draw[dashed] (m-13-18.north east) -- (m-6-20.north east);
\draw (m-3-18.north west) -- (m-3-21.north east) -- (m-6-21.north east) -- (m-6-21.north west);

 \draw[decorate] (m-3-25.north west) -- (m-10-25.south west);
 \draw (m-10-25.south west) -- (m-10-25.south east) -- (m-10-25.north east);
\draw[dashed] (m-10-25.north east) -- (m-6-27.north east);
\draw (m-3-25.north west) -- (m-3-28.north east) -- (m-6-28.north
east) -- (m-6-28.north west);
\draw (m-1-17) -- (m-23-17);
\draw (m-1-3) -- (m-21-3.south) -- +(0cm, -1cm);
 \end{tikzpicture}
 \caption{Pictorial Construction for $k_1$ and $k_2$}
 \label{fig:gen-pic-const}
 \end{sidewaysfigure}
In this tableaux diagram for $k_1k_2$
the vectors attached to the $\ell$th row are  $k_2 \mathfrak{P}_\ell :=
k_2 \mathfrak{p}^{(i_1)}(\overline{t}^{k_2})
  \mathfrak{q}^{(j_1)} (\overline{t}^{k_1})$ and $k_1
  \mathfrak{Q}_{\ell} := k_1
  \mathfrak{p}^{(i_2)}(\overline{t}^{k_2}) \mathfrak{q}^{(j_2)}
  (\overline{t}^{k_1})$ where $(i_1, j_1)$ and $(i_2, j_2)$ are
  uniquely determined by: 
  \begin{align} \label{eq:index-split}
    \left\{
    \begin{aligned}
    \ell = k_2i_1 &+ j_1 = k_1 j_2 + i_2\\
    0&\leqslant j_s \leqslant k_{3-s}-1, \quad s \in \{1, 2\}.
    \end{aligned}
    \right.
  \end{align}

{\bfseries Observations.} We now make the following observations about
the construction: 
\begin{enumerate}[ref=(\arabic*)]
\item The diagram in Figure~\ref{fig:gen-pic-const} is the tableaux diagram
  associated to $k_1k_2$. 
\item \label{obsns-pt-2}For $0 \leqslant \ell \leqslant k_1k_2 -1$, we have
  $e_{k_1k_2}(\ell) = e_{k_1}(i_1)e_{k_2}(j_2)$ where $i_1$ and $j_2$
  are determined by \eqref{eq:index-split}.
\begin{proof}
The $\ell$th elementary divisor of $k_1k_2$ is the product of the
entries in the $\ell$th row of Figure~\ref{fig:gen-pic-const}: clearly, the tableaux diagram for $k_1$ contributes $e_{k_1}(i_1)$ and
that of $k_2$ contributes $e_{k_2}(j_2)$. 
\end{proof}
\item \label{obsns-pt-3}With notations as in \ref{obsns-pt-2}, if we can find $d_{1,
    \ell}, d_{2, \ell}$ such that $\gcd(d_{1, \ell}, k_1) = \gcd(d_{2,
    \ell}, k_2) =1$ and
  $\{d_{1,\ell}k_2\mathfrak{P}_\ell+d_{2,\ell}k_1\mathfrak{Q}_\ell:0
  \leqslant \ell \leqslant k_1k_2-1\}$ is a basis for the codomain of
  $\Psi_{k_1k_2}$, then, $(d_{1,\ell}k_2\mathfrak{P}_\ell+d_{2,\ell}k_1\mathfrak{Q}_\ell:0
  \leqslant \ell \leqslant k_1k_2-1)$ is a Smith vector for $k_1k_2$. 
\begin{proof}
From Lemma~\ref{lem:tensor-prod-of-smith-vectors-works}, it is clear that the order of
$\overline{\Psi}_{k_1k_2}(k_1\mathfrak{P}_\ell)$ is $e_{k_1}(i_1)$ and that of
$\overline{\Psi}_{k_1k_2}(k_2\mathfrak{Q}_\ell)$ is $e_{k_2}(j_2)$. Since $\gcd(d_{s, \ell},
k_s) = 1$, the order of
$\overline{\Psi}_{k_1k_2}(k_s(\cdot)_\ell)$ remains unchanged after multiplication by $d_{s,
  \ell}$. Therefore, the order of
$\overline{\Psi}_{k_1k_2}(d_{1,\ell}k_1\mathfrak{P}_\ell+d_{2,\ell}k_2\mathfrak{Q}_\ell)$
equals $e_{k_1}(i_1)e_{k_2}(j_2)$. This equals $e_{k_1k_2}(\ell)$, by
\ref{obsns-pt-2} and we are done.  
\end{proof}
\item Define the following bijections between $[k_1k_2]$ and $[k_1]
  \times [k_2]$ where $[n] := \{0, 1,
  \dots, n-1\}$:
  \begin{align*}
    \iota(\ell) &= (i_1, j_1) \\
    \iota'(\ell) &= (i_2, j_2)  
  \end{align*}
where $(i_1, j_1)$ and $(i_2, j_2)$ are determined from $\ell$ using
\eqref{eq:index-split}. Let $L$ and $L'$ be total orderings on $[k_1] \times [k_2]$ defined by
transferring the order from $[k_1k_2]$ via $\iota$ and $\iota'$
respectively. Then, $L$ is in lexicographic order and $L'$ is in
reverse lexicographic order. Let $\sigma_{k_1, k_2}$ the
permutation $\iota'^{-1}\iota$ of the set $[k_1k_2]$.
\begin{eg}
  Take $k_1=2, k_2=3$. Then, $L$ and $L'$ are increasing along rows in
  the table below. 
  \begin{table}[htbp]
    \centering
    \begin{tabular}{c|cccccc}\hline
      $\ell$ & 0 & 1 & 2 & 3 & 4 & 5\\ \hline\hline
      $L$    & (0, 0) & (0, 1) & (0, 2) & (1, 0) & (1, 1) & (1, 2)\\\hline
      $L'$   & (0, 0) & (1, 0) & (0, 1) & (1, 1) & (0, 2) & (1, 2) \\ \hline
    \end{tabular}
  \end{table}
  In cycle notation, $\sigma_{2, 3} \equiv (1\ 2\ 4\ 3)$. 
\end{eg}
\item \label{obsn-pt-5}Let $\mathfrak{P} = (\mathfrak{P}_\ell : 0 \leqslant \ell
  \leqslant k_1k_2 -1)$ and $P$ be the permutation matrix
  $\sigma_{k_1, k_2}I$. By Lemma~\ref{lem:diagonals-condn-can-be-met},
  there are diagonal matrices $D_1$ and $D_2$ such that $\det(k_2D_1 +
  k_1D_2P) = 1$. Since $\{\mathfrak{P}_\ell: 0 \leqslant \ell
  \leqslant k_1k_2 - 1\}$ is a basis for the codomain of
  $\Psi_{k_1k_2}$, we have that the components of $(k_2D_1 +
  k_1D_2P)\mathfrak{P}^t$  forms a basis for the codomain of
  $\Psi_{k_1k_2}$. Now, setting $D_1 = \diag(d_{1, 1}, \dots, d_{1,
    k_1k_2-1})$ and $D_2 = \diag(d_{2, 1}, \dots, d_{2, k_1k_2-1})$, by
  \ref{obsns-pt-3}, we have that 
\[\operatorname{TSV}(k_1, k_2, \mathfrak{p}, \mathfrak{q}) := (d_{1,\ell}k_2\mathfrak{P}_\ell + d_{2,\ell} k_1 \mathfrak{Q}_\ell:
0 \leqslant \ell \leqslant k_1k_2-1)\]
is a Smith vector for $k_1k_2$. 
\end{enumerate}

\begin{algo}
Given the above observations, we are now ready
to present the algorithm for this module. 

\begin{description}
\item[Step 1] Construct the permutation matrix $P = \sigma_{k_1,
  k_2}I$. 
\item[Step 2] Construct diagonal matrices $D_1$ and
  $D_2$ as in the proof of Lemma~\ref{lem:diagonals-condn-can-be-met}.
\item[Step 3] Construct the vector $\mathfrak{P} = (\mathfrak{P}_\ell:
  0 \leqslant \ell \leqslant k_1k_2-1)$ by the formula:
\[\mathfrak{P}_\ell = \mathfrak{p}^{(i_1)}(\overline{t}^{k_2}) \mathfrak{q}^{(j_1)}(\overline{t}^{k_1}) \]
where $i, j$ and $\ell$ are related by \eqref{eq:index-split}. 
\item[Step 4] From the entries of the vector $k_2D_1 \mathfrak{P}^t+k_1D_2P \mathfrak{P}^t$,
  contruct the Smith vector $\operatorname{TSV}(k_1, k_2,
  \mathfrak{p}, \mathfrak{q})$ as in observation \ref{obsn-pt-5}. 
\end{description}
\end{algo}

{\bfseries Example.} Let us consider the example $k_1 = 2, k_2 =
3, \mathfrak{p} = \text{SV}(2)$ and $\mathfrak{q} = \text{SV}(3)$. Figure~\ref{fig:23} illustrates the construction we carried out
in this subsection. 
\begin{figure}[htbp]
\begin{tikzpicture}
\matrix[matrix of math nodes, execute at empty cell =
\node{\phantom{\oplus}};] (m){
&                           &[0.2cm]&k_1=2&[0.1cm]&&[0.1cm]&k_2=3&[0.1cm]&&\\[0.4cm]
\phantom{\oplus}&\text{Row Index} &[0.2cm]\phantom{\oplus}& \phantom{\oplus} &&&\phantom{\oplus} &&&[3cm] &\scriptstyle\text{TSV}(k_1, k_2,\mathfrak{p}, \mathfrak{q})&\phantom{\oplus}\\
\phantom{\oplus}&5                         &[0.2cm]&2&&\scriptstyle 3(\overline{1} \oplus \overline{1}
\oplus \overline{1} \oplus \overline{1})&&3&&\scriptstyle 2(\overline{1} \oplus \overline{1}
\oplus \overline{1} \oplus \overline{1}) &\scriptstyle \overline{1}\oplus \overline{1} \oplus \overline{1} \oplus \overline{1}&\phantom{\oplus}\\
\phantom{\oplus}&4                         &[0.2cm]&2&&\scriptstyle 3(\overline{1} \oplus \overline{1}
\oplus \overline{t}^2 \oplus \overline{t}^2)&&3&&\scriptstyle 2(\overline{1} \oplus \overline{1}
\oplus \overline{0} \oplus \overline{0})  &\scriptstyle \overline{5}\oplus\overline{3}\oplus-3\overline{t}-\overline{1}\oplus 3\overline{t}-\overline{3}&\phantom{\oplus}\\ 
\phantom{\oplus}&3                         &[0.2cm]&2&&\scriptstyle 3(\overline{1} \oplus \overline{1}
\oplus \overline{0} \oplus \overline{0})& & &&\scriptstyle 2(\overline{1}\oplus \overline{0} \oplus
\overline{t}^2 \oplus \overline{0})&\scriptstyle \overline{5} \oplus \overline{5} \oplus -2\overline{t}-\overline{2}\oplus 2\overline{t}-\overline{2}&\phantom{\oplus}\\
\phantom{\oplus}&2                         &[0.2cm]&
&&\scriptstyle 3(\overline{1}\oplus\overline{0}\oplus\overline{1}\oplus\overline{0})& &  &&\scriptstyle 2(\overline{1} \oplus \overline{1}
\oplus \overline{t}^2 \oplus \overline{t}^2)&\scriptstyle \overline{5} \oplus \overline{0} \oplus -2\overline{t}+\overline{1}
\oplus \overline{0}&\phantom{\oplus}\\
\phantom{\oplus}&1                         &[0.2cm]&  &&\scriptstyle 3(\overline{1}\oplus \overline{0} \oplus
\overline{t}^2 \oplus \overline{0})& & &&\scriptstyle
2(\overline{1}\oplus\overline{0}\oplus\overline{1}\oplus\overline{0})&\scriptstyle\overline{13}\oplus \overline{10} \oplus
-3\overline{t}-\overline{3} \oplus \overline{0}&\phantom{\oplus}\\
\phantom{\oplus}&0                         &[0.2cm]\phantom{\oplus}&  &&\scriptstyle
3(\overline{1}\oplus\overline{0}\oplus\overline{0}\oplus\overline{0})& \phantom{\oplus}& \phantom{\oplus}&\phantom{\oplus}&\scriptstyle
2(\overline{1}\oplus\overline{0}\oplus\overline{0}\oplus\overline{0}) &\scriptstyle\overline{1}\oplus \overline{0}\oplus \overline{0}
\oplus \overline{0}&\phantom{\oplus}\\
};
\foreach \i in {3,...,8} {
\draw (m-\i-1.south west) -- (m-\i-12.south east);
};
\draw (m-2-1.south west) -- (m-2-12.south east);
\draw ([yshift=0.5mm]m-3-4.north west) -- (m-5-4.south west);
\draw ([yshift=0.5mm]m-3-4.north east) -- (m-5-4.south east);
\draw ([yshift=0.5mm]m-3-8.north west) -- (m-4-8.south west);
\draw ([yshift=0.5mm]m-3-8.north east) -- (m-4-8.south east);
\draw (m-2-3) -- (m-8-3.south);
\draw (m-2-7) -- (m-8-7.south);
\draw ([xshift=1mm]m-3-10.north east) -- ([xshift=1mm]m-8-10.south east);
\end{tikzpicture}
\caption{Pictorial Construction for $k_1 = 2$ and $k_2 = 3$}
\label{fig:23}
\end{figure}

\subsection{$\boldsymbol{SV(n)}$}This algorithm is a recursion using
the modules $\boldsymbol{SV(p^e)}$ and $\boldsymbol{TSV(k_1,k_2,
  \mathfrak{p}, \mathfrak{q})}$. 

\begin{algo}\mbox{}
\begin{description}
\item[Step 1] Factorise $n$ as $p_1^{e_1} \dots p_r^{e_r}$. For each
  $p_i^{e_i}$, calculate $\operatorname{SV}(p_i^{e_i})$. 
\item[Step 2] Having calculated $\operatorname{SV}(p_1^{e_1} \dots
  p_i^{e_i})$, we calculate $\operatorname{SV}(p_1^{e_1} \dots
  p_{i+1}^{e_{i+1}})$ by the formula 
\[\operatorname{SV}(p_1^{e_1} \dots
  p_{i+1}^{e_{i+1}}) = \operatorname{TSV}(p_1^{e_1} \dots
  p_i^{e_i}, p_{i+1}^{e_{i+1}},  \operatorname{SV}(p_1^{e_1} \dots
  p_i^{e_i}), SV(p_{i+1}^{e_{i+1}})).\] 
\end{description} 
\end{algo}

\subsection{Analysis of Algorithms}
\label{ssec:time-complexity}
To calculate the number of bit operations needed to output
$\operatorname{SV}(n)$, we compute the bit complexities of the modules
\[\boldsymbol{TSV(p_1^{e_1}\dots p_\ell^{e_\ell}, p_{\ell +1}^{e_{\ell+1}}
  \operatorname{SV}(p_1^{e_1}\dots p_\ell^{e_{\ell+1}}),
  \operatorname{SV}(p_{\ell +1}^{e_{\ell+1}}))}\]
for $l = 0, \dots, r-1$, where $n = p_1^{e_1}\dots p_r^{e_r}$ is the
prime factorisation of $n$. We observe that the bit complexity in the case
$l = r-1$ dominates and so the total bit complexity is atmost $r$ times
this complexity. 

Firstly, we need the following definitions and some
lemmas. 

Recall that for a polynomial $a(X) = \sum_{i=0}^n a_i X^i$ with integer
coefficients, its height $\hgt(a)$ is defined to be $\max\{|a_i|: 0 \leqslant i
\leqslant n\}$.   
\begin{defn}
Given a vector $(\mathfrak{p}^{(j)}: 0 \leqslant j \leqslant
n-1)$ having entries in $\bigoplus_{d\mid n}\ZZ[X]/\Phi_d(X)$ such that $\mathfrak{p}^{(j)} = \bigoplus_{d \mid n}
p_d^{(j)} \bmod{\Phi_d(X)}$ where $p_d^{(j)}$ is the unique
representative of degree atmost $\phi(d)-1$, we define its height to
be: 
\[\hgt((\mathfrak{p}^{(j)}: 0 \leqslant j \leqslant n-1)) =
\max\{\hgt(p_d^{(j)}): d \mid n, 0 \leqslant j \leqslant n-1\}.\] 
\end{defn}
\begin{note}
\label{note:nota-and-cpty}
Given a positive integer $n$, denote the bit
length of $n$ by $\mathcal{B}(n)$. Let $\mu(n_1, n_2)$ denote the
number of bit operations required to multiply a $n_1$-bit number with
an $n_2$-bit number. We will also set 
$\mu(n):=\mu(n, n)$.
\begin{itemize}
\item[(i)] We have $\mu(n_1, n_2) \leq \mu(n)$ where $n = \max(n_1, n_2)$. The standard multiplication
would suggest $\mu(n) = O( n^2)$ while FFT-methods in Sch{\"o}nage-Strassen algorithm \cite{SchStr} indicate that
$\mu(n) = O(n\log n \log\log n)$. This bound was
later improved by F{\"u}rer \cite{Furer:2007aa}. In contrast, addition
of a $n_1$-bit number to a $n_2$-bit number takes $O(n_1+n_2)$ bit operations. 
\item[(ii)] Recall that in Step~1 of the proof of
  Lemma~\ref{lem:diagonals-condn-can-be-met}, we calculate integers
  $u$ and $v$ such that $um^k + v(-1)^{k-1}n^k = 1$. By extended
  Euclidean algorithm, this computation takes $O(k^2\log(n)\log(m))$ bit
  operations. Also, $\mathcal{B}(u)$ and $\mathcal{B}(v)$ are
  $O(k\max(\log(n), \log(m)))$ (see \cite[Section~2.2]{Yap}).
\item[(iii)] Let $a(X), b(X) \in \mathbf{Z}[X]$ of degree $m$ and
  $n$ with bit length of the heights equal to $\tau_1$ and $\tau_2$
  respectively.  The standard polynomial multiplication algorithm to
  multiply $a(X)$ and $b(X)$ takes
  $O(mn\mu(\tau_1,\tau_2))$ bit operations. However, FFT based
  algorithms improves this to
  $O(d\log(d)\mu(\tau_1+\tau_2+\log d))$ where $d = \max(m, n)$. For more details, see \cite[Corollary~8.27]{MCA} and \cite[Lemma~17]{PanTsig14}.  
\end{itemize}
\end{note}  
We begin with some lemmas:
\begin{lem}
\label{lem:ht-of-quorem}
Let $a(X), b(X) \in \ZZ[X]$ be such that $b(X)$ is monic and $\deg(a)
\geqslant \deg(b)$. Suppose that $\hgt(a) = M_1$ and $\hgt(b) =
M_2$. Let $a(X) = b(X)q(X) + r(X)$ where $q(X)$ and $r(X)$ are
quotient and remainder respectively. Then, 
\[\hgt(r(X)) \leqslant M_1 (1+M_2)^{\deg(a) - \deg(b)+1}.\]
\end{lem} 
Therefore, we have that $\mathcal{B}(\hgt(r(X)) = O(\log(M_1) +
(\deg(a))\log(M_2))$. 
\begin{proof}
  In each step of the Standard division algorithm, we subtract a
  multiple of $b(X)$ from $a(X)$ so as to reduce the degree of
  $a(X)$. Suppose that at the $i$th stage, we are left with a
  polynomial $a_i(X)$ of height $h_i$ where $0 \leqslant i \leqslant \deg(a) -
  \deg(b) +1$. Then, $h_i \leqslant h_{i-1} + M_2 h_{i-1}$ with $h_0 =
  M_1$. So, the height of $r(X)$ is atmost $M_1(1+M_2)^{\deg(a)-\deg(b)+1}$.
\end{proof}
\begin{lem}[{\cite[Theorem~21]{PanTsig14}}]
\label{lem:bit-comp-for-div}
 Let $a(X), b(X) \in \ZZ[X]$ be such that $b(X)$ is monic of degree
 $n$ and $\deg(a) \leqslant 2 \deg(b)$. Set $\mathcal{B}(\hgt(a)) = \tau_1$ and
 $\mathcal{B}(\hgt(b)) = \tau_2$. Then, the number of bit operations
 required to compute the quotient and remainder on dividing $a(X)$ by
 $b(X)$ is $O(n\log^2(n) \mu(n\tau_2 + \tau_1))$. 
\end{lem}
\begin{lem}
\label{lem:bit-comp-for-div-1}
Let $a(X), b(X) \in \ZZ[X]$ be of degree $m$ and $n$ respectively with
$m \geqslant n$. Suppose also that $b$ is monic. Set $\mathcal{B}(\hgt(a)) = \tau_1$ and
 $\mathcal{B}(\hgt(b)) = \tau_2$. Then, the number of bit operations
 required to compute the quotient and remainder on dividing $a(X)$ by
 $b(X)$ is $O(m\log^2(m) \mu(m\tau_2 + \tau_1))$. 
\end{lem}
\begin{proof}
  If $m \leqslant 2n$, then the assertion follows from the
  Lemma~\ref{lem:bit-comp-for-div}. 

 So, let $m > 2n$. Choose the least $k_1$ such that the polynomials $a(X)$ and
 $X^{k_1}b(X)$ satisfy the hypotheses of the Lemma~\ref{lem:bit-comp-for-div}, that is,
  \[2(k_1+ n) \geqslant m \geqslant 2(k_1 - 1 + n).\] Now,
  proceeding as in Lemma~\ref{lem:bit-comp-for-div}, we may obtain
  polynomials $q_1(X)$ and $r_1(X)$ with $\deg(r_1) < k_1 + n$ such that 
  $a(X) = X^{k_1}b(X) q_1(X) + r_1(X)$.
  Now, if $\deg(r_1) < n$, then, $r_1$ is indeed the remainder on
  dividing $a(X)$ by $b(X)$. If not, we divide $r_1(X)$ by
  $X^{k_2}b(X)$ for $k_2$ chosen so that 
  \[2(k_2+n) \geqslant \deg(r_1) \geqslant 2(k_2+n-1).\]
  Continuing this way, we obtain the remainder $r(X)$ on dividing
  $a(X)$ by $b(X)$. 

  First note that $\deg(r_i) \leqslant k_i + n -1$. Since $k_1 + n - 1
  \leqslant \frac m 2$, and $k_i + n - 1 \leqslant
  \frac{k_{i-1}+n-1}2$ for all $i \geqslant 2$, we get
  \[\deg(r_i) \leqslant k_i+n-1 \leqslant \frac m{2^i}.\]

  Secondly, using Lemma~\ref{lem:ht-of-quorem}, we have: 
  \begin{align*}
    \mathcal{B}(\hgt(r_i)) &=
    O\bigg(\mathcal{B}(\hgt(r_{i-1}))+(\deg(r_{i-1})-(k_i+n)+1)\tau_2\bigg)\\
    &=O\bigg(\mathcal{B}(\hgt(r_{i-1}))+\frac{m}{2^{i-1}}\tau_2\bigg) \\
    &=O(\tau_2m+\tau_1)
  \end{align*}
  Using Lemma~\ref{lem:bit-comp-for-div}, the bit
  complexity in obtaining $q_i$ and $r_i$ is
  \[O((n+k_i)\log^2(n+k_i)\mu((n+k_i)\tau_2+m\tau_2+\tau_1)).\] 
  Thus, the total bit complexity is: 
  \begin{align*}
    &O\left(\sum_{i \geqslant 1}
      (n+k_i)\log^2(n+k_i)\mu((n+k_i)\tau_2+m\tau_2+\tau_1)\right) \\
    &= O\left(\sum_{i \geqslant 1} \frac m {2^i}\log^2(m)\mu(2m\tau_2+\tau_1)\right)\\
    &=O(m\log^2(m)\mu(m\tau_2+\tau_1)).
  \end{align*}
This completes the proof. 
\end{proof}

We need the following estimate for coefficients of the cyclotomic
polynomials due to Bateman \cite{Bman}.
\begin{lem}[Bateman]
\label{lem:cyclo-coeff}
The height of the cyclotomic polynomial $\Phi_n$ is $O(\exp(n^{C/\log\log
  n}))$ for some absolute constant $C$.
\end{lem}
Thus, $\mathcal{B}(\hgt(\Phi_{n}))$ is $O(n^{\epsilon})$ for any
$\epsilon > 0$. 
\begin{lem}
\label{lem:smith-coeff}
The bit length of the height of $\operatorname{SV}(n)$ is
$O(n^{1+\epsilon})$ for any $\epsilon > 0$. 
\end{lem}
\begin{proof}
  Suppose that the prime factorisation of $n$ is $p_1^{e_1} \dots
  p_r^{e_r}$. We prove this result by induction on $r$. By
  Remark~\ref{rem:ht-of-prime-power-smith}, 
we have that 
\[\hgt(\operatorname{SV}(p_1^{e_1})) = O(1),\]
and therefore, the bit length
$\mathcal{B}(\hgt(\operatorname{SV}(p_1^{e_1}))$ of the height is
$O(1)$. 

Let $\mathfrak{p} :=
\operatorname{SV}\left(\frac{n}{p_r^{e_r}}\right)$ and $\mathfrak{q}
:= \operatorname{SV}(p_r^{e_r})$. The bit length of the height of the
polynomial
$\mathfrak{p}_{d_1}^{(i)}(t^{p_r^{e_r}})\mathfrak{q}_{d_2}^{(j)}(t^{np_r^{-e_r}})$
(before reducing modulo $\Phi_{d_1d_2}$) is \[O\left(\vphantom{\frac12}\log(\max_{j,
  d_2}(\deg(q_{d_2}^{(j)}))) + \mathcal{B}(\hgt(\mathfrak{p}))\right)\] since
the coefficients in the product is given by convolution.   By
induction hypothesis for the height of $\mathfrak{p}$, this
equals $O(\log p_r^{e_r} + (np_r^{-e_r})^{1+\epsilon})$ for any
$\epsilon > 0$. 

Using Lemmas~\ref{lem:ht-of-quorem} and \ref{lem:cyclo-coeff}, after
reduction $\bmod{\Phi_{d_1d_2}}$, we get:
\begin{align*}
  \mathcal{B}(\hgt(\mathfrak{p}_{d_1}^{(i)}(\overline{t}^{p_r^{e_r}})\mathfrak{q}_{d_2}^{(j)}(\overline{t}^{np_r^{-e_r}})))
  &= O(\log p_r^{e_r} + (np_r^{-e_r})^{1+\epsilon} + n\cdot
  n^{\epsilon}) \\ &= O(n^{1+\epsilon})
\end{align*}
Now, scaling the coefficients of $\mathfrak{p}$ and $\mathfrak{q}$ by
$p_r^{e_r}D_1$ and $np_r^{-e_r}D_2P$ as in Observation \ref{obsn-pt-5} of
Section~\ref{ssec:meat-of-algo} contributes to an addition of atmost
$O(n\log n)$ bits to the total height (see Note~\ref{note:nota-and-cpty}(ii)). This finishes the proof.  
\end{proof}

Now, we are ready to calculate the bit complexity of each of these
modules. We will use the soft-Oh notation $O^\sim(\cdot)$ which drops
out polylogarithmic factors. For functions $f, g: \mathbf{R}^s \to \mathbf{R}$, we say that $f
= O^\sim(g)$ if there is a constant $c > 0$ such that $f = O(g
\log^c(g))$. We will let $\epsilon$ be an arbitrary positive real
number. \newline

{\bfseries $\boldsymbol{SV(p^e)}$.} Since $W_{p^e}$ is given by a
recursive formula, calculation of Smith vector for $p^e$ takes
$O(p^{2e})$ steps.  

{\bfseries $\boldsymbol{TSV(k_1, k_2, \mathfrak{p}, \mathfrak{q})}$.}
We will calculate the bit complexity of this algorithm by going over
each step of the algorithm:
\begin{itemize}
\item In Step~1, the algorithm needs $O((k_1k_2)^2)$ steps. 
\item For Step~2, in
determining the diagonal matrices $D_1$ and $D_2$ as in Lemma~\ref{lem:diagonals-condn-can-be-met}, we need the cycles of
the permutation $\sigma_{k_1, k_2}$. The construction of
$\sigma_{k_1k_2}$ and its cycle decomposition takes $O((k_1k_2)^2)$
bit operations. Suppose that the cycle lengths of
$\sigma$ are $c_1, \dots, c_s$ so that $\sum_{i=1}^s c_i =
k_1k_2$. By Note~\ref{note:nota-and-cpty}~(ii), the bit complexity of this step is 
\[O\left((k_1k_2)^2+\sum_{i=1}^s c_i^2 \log(k_1) \log(k_2)\right) =
O^\sim\left((k_1k_2)^2\right).\] 
\item In Step~3, to construct $\mathfrak{P}_\ell$, we need to find
$\mathfrak{p}_{d_1}^{(i_1)}(t^{k_2})\mathfrak{q}_{d_2}^{(j_1)}(t^{k_1})
\bmod{\Phi_{d_1d_2}(t)}$ for every $d_1 \mid k_1$ and $d_2 \mid
k_2$. Since $\mathfrak{p}_{d_1}^{(i_1)}$ and
$\mathfrak{q}_{d_2}^{(j_1)}$ has degree atmost $\phi(d_1)$ and $\phi(d_2)$
respectively, the standard polynomial multiplication (Note~\ref{note:nota-and-cpty}~(iii))
costs $O(\phi(d_1d_2)(\mu_h(\mathfrak{p}, \mathfrak{q})+\log(\phi(d_1d_2))))$ bit operations where
\[\mu_h(\mathfrak{p}, \mathfrak{q} ) :=
\mu(\mathcal{B}(\hgt(\mathfrak{p})),
\mathcal{B}(\hgt(\mathfrak{q}))).\] To reduce modulo
$\Phi_{d_1d_2}(t)$, using Lemma~\ref{lem:bit-comp-for-div-1} and Lemma~\ref{lem:cyclo-coeff}  
we need 
$O^\sim\big(\beta(d_1, d_2)^2(d_1d_2)^\epsilon+\beta(d_1, d_2)\mu_h(\mathfrak{p},
\mathfrak{q})\big)$
bit operations where $\beta(d_1, d_2) =
k_1\phi(d_2)+k_2\phi(d_1)$. Since $\phi(d_1d_2) \leqslant \beta(d_1,
d_2)$, the total bit complexity involved in finding
$\mathfrak{p}_{d_1}^{(i_1)}(t^{k_2})\mathfrak{q}_{d_2}^{(j_1)}(t^{k_1})
\bmod{\Phi_{d_1d_2}(t)}$ is
\begin{align*}
O^\sim\big(\beta(d_1, d_2)^2(d_1d_2)^\epsilon+\beta(d_1, d_2)\mu_h(\mathfrak{p},
\mathfrak{q})\big). 
\end{align*}
Now, we have:
\begin{align*}
\sum_{\substack{d_1 \mid k_1 \\ d_2 \mid k_2}} \beta(d_1, d_2) &=
k_1k_2(\tau(k_1)+\tau(k_2)) \\
&= O((k_1k_2)^{1+\epsilon}) \\
\intertext{ since $\tau(n)$ (the number of divisors of $n$) is
  $O(n^{1+\epsilon})$ (\cite[Theorem~315]{hardy1979introduction}). Also}
\sum_{\substack{d_1 \mid k_1 \\ d_2 \mid k_2}} \beta(d_1, d_2)^2
&\leqslant \Big(\sum_{\substack{d_1 \mid k_1 \\ d_2 \mid k_2}} \beta(d_1, d_2)\Big)^2.
\end{align*}

Thus, the bit complexity to compute $\mathfrak{P}_\ell$ is:   
\begin{align*}
  &\quad O^\sim\Big(\sum_{\substack{d_1 \mid k_1 \\ d_2
        \mid k_2}} \beta(d_1, d_2)^2(d_1d_2)^\epsilon+\beta(d_1, d_2)\mu_h(\mathfrak{p},
\mathfrak{q})\Big) \\
&= O^\sim\Big((k_1k_2)^\epsilon\sum_{\substack{d_1 \mid k_1 \\ d_2
        \mid k_2}} \beta(d_1, d_2)^2+\mu_h(\mathfrak{p},
\mathfrak{q})\sum_{\substack{d_1 \mid k_1 \\ d_2
        \mid k_2}}\beta(d_1, d_2)\Big) \\
&= O^\sim\Big((k_1k_2)^{2+\epsilon}+\mu_h(\mathfrak{p}, \mathfrak{q})(k_1k_2)^{1+\epsilon}\Big)
\end{align*}
So the total bit complexity in Step~3 is
  \begin{align*}
    O^\sim((k_1k_2)^{3+\epsilon}+\mu_h(\mathfrak{p},
    \mathfrak{q})(k_1k_2)^{2+\epsilon}).
    \end{align*}
\item Using the facts in Note~\ref{note:nota-and-cpty}~(ii), one may prove that the bit lengths of the entries of the matrices $k_2D_1$ and
  $k_1D_2$ are $O^\sim(k_1k_2)$. Since there are
  $(k_1k_2)^2$ terms in $\vec{\mathfrak{P}}$, calculating $(k_2D_1 +
  k_1D_2P)\mathfrak{P}^t$, which is the Step~4, needs
  $O^\sim((k_1k_2)^{2+\epsilon}\mu_h(\mathfrak{p}, \mathfrak{q})+(k_1k_2)^{3+\epsilon})$ bit operations. 
\end{itemize} 
Summing the bit complexities of each step, we conclude that the bit
complexity of this module is 
\begin{equation}
\label{eq:tsv}
O^\sim((k_1k_2)^{2+\epsilon}\mu_h(\mathfrak{p},
\mathfrak{q})+(k_1k_2)^{3+\epsilon}) \text{ for any } \epsilon >
0. 
\end{equation}
{\bfseries $\boldsymbol{SV(n)}$.}  The bit complexity of Step~1 is
$O(n^2)$. For Step~2, the maximum bit operations are required in the last step of the
recursion. We know that the bit complexity of module
$\boldsymbol{TSV(\frac{n}{p_r^{e_r}}, p_r^{e_r},
  \operatorname{SV}(\frac{n}{p_r^{e_r}}), \operatorname{SV}(p_{r}^{e_{r}}))}$ is:
\begin{align}
\label{eq:sv-n-cplxty}
  &\quad O^\sim\left(n^{2+\epsilon}\mu_h\Big(\operatorname{SV}\big(\frac{n}{p_r^{e_r}}\big),\operatorname{SV}(p_r^{e_r})\Big)+n^{3+\epsilon}\right)
  && \text{using } \eqref{eq:tsv}
  \notag\\ 
  &=
  O^\sim\left(n^{2+\epsilon}\mu\Big(\frac{n}{p_r^{e_r}},p_r^{e_r}\Big)+n^{3+\epsilon}\right)
 && \text{by Lemma~\ref{lem:smith-coeff}}
  \notag \\
  &= O^\sim(n^{3+\epsilon}) 
\end{align}
Since $r = O(\log(n))$, the bit
complexity of the module
$\boldsymbol{SV(n)}$ is $O(n^{3+\epsilon})$ for any $\epsilon > 0$.  

\subsubsection{Space Complexity}
Note that the bit length of an integer that appears while executing
the module $\boldsymbol{SV(n)}$ is $O(n^{1+\epsilon})$ for any
$\epsilon > 0$. The maximum
space is needed to store a Smith vector for $n$. Since a Smith vector has $O(n^2)$
terms, the space complexity of $\operatorname{SV}(n)$ is
$O(n^{3+\epsilon})$ for any $\epsilon > 0$. 

We summarise the above discussion in the following theorem:
\begin{thm}
\label{thm:main-algo}
  Given a positive integer $n$, the algorithm $\boldsymbol{SV(n)}$
  gives a Smith vector for $n$. The bit complexity and space
  complexity of this
  algorithm are both $O(n^{3+\epsilon})$ for any $\epsilon > 0$. 
\end{thm}

We may compare this theorem with the results in the literature in this
direction. It appears to us that the best known algorithm for
determining Smith normal form $S(A)$ of an integer matrix $A$ and the
unimodular transforming matrices $U$ and $V$ are due to Arne Storjohann 
in his PhD thesis \cite{ASThesis}. Let $O(n^\theta)$ be the algebraic
complexity involved in multiplying two $n \times n$ matrices with
integer entries; best known algorithms give $2 < \theta \leqslant 3$
(for example,  V. Vassilevska Williams gives an algorithm with $\theta = 2.373$
in \cite{MMultVV}). He proves:

\begin{thm*}
[{\cite[Proposition~7.20]{ASThesis}}] 
For a $n \times m$ matrix $A=(a_{ij})$ of rank $r$ with integer entries, the Smith normal form $S(A)$ and the
unimodular transforming matrices $U$ and $V$ may be obtained in
$O^\sim(nmr^{\theta-1}\log \|A\| + nm\mu(r\log\|A\|))$ bit
operations where $\|A\| = \max_{i,j} |a_{ij}|$. 
\end{thm*}

Specialising to our case, it is seen that Storjohann's algorithm would
require $O(n^{2+\theta+\epsilon})$ bit operations. Here we use 
$\log\|A\|=O(n^{1+\epsilon})$ from Lemma~\ref{lem:cyclo-coeff} and Lemma~\ref{lem:ht-of-quorem}. Thus, our
algorithm is an improvement to this best known algorithm in the
special case we are interested in. 
\appendix 
\section{Determinant of $A_n$}
\label{sec:det-of-an}
We now calculate the determinant of the matrix $A_n$ in terms of
resultants of cyclotomic polynomials. The advantage of this new
approach is the fact that this generalises to any monic polynomial $f$
over a unique factorisation domain (UFD) and
any of its factorisations into pairwise relatively prime
polynomials. Moreover, this approach determines the sign of
$\det(A_n)$ unambigously. 

To calculate the determinant of $A_n$, we propose the following
simplification: 
\subsection{Simplification}
\label{ssec:a-tech}
Let $\Omega_n$ be the cyclic group of all $n$th roots of
unity. The cyclotomic polynomials $\{\Phi_d\}_{d \mid n}$ factorise
over the ring $\ZZ[\Omega_n]$ of cyclotomic integers. We will see that
calculating $\det(A_n)$ becomes very simple when we work over the ring of
cyclotomic integers. Now, we  shall argue that passing to
$\ZZ[\Omega_n]$ does not affect the
computation. 

To see this, consider the following diagram of maps:
\[
\begin{tikzcd}
\frac{\ZZ[X]}{(X^n - 1)} \ar{d}[swap]{\Psi_n} \arrow[hook]{r}{\iota_1} &
\frac{\ZZ[\Omega_n][X]}{(X^n-1)} \ar{d}{\widetilde{\Psi}_n} \arrow{dr}[description]{\rho_2}\\
\bigoplus_{d \mid n} \frac{\ZZ[X]}{\Phi_d(X)} \arrow[hook]{r}{\iota_2} 
& \bigoplus_{d \mid n}\frac{\ZZ[\Omega_n][X]}{(\Phi_d(X))} \ar{r}{\rho_1}&  \bigoplus_{\omega \in
  \Omega_n}\frac{\ZZ[\Omega_n][X]}{(X-\omega)}
\end{tikzcd}
\]
where $\widetilde{\Psi}_n$ is the $\ZZ[\Omega_n]$-linear extension of
$\Psi_n$,  $\rho_1$ is the canonical quotient map, 
$\rho_2$ is the composition $\rho_1 \circ
   \widetilde{\Psi}_n$ and $\iota_1, \iota_2$ are canonical
  inclusions. 
The matrix of $\widetilde{\Psi}_n$ with respect to $(1, \overline{X}, \dots,
 {\overline{X}}^{n-1})$ as a basis for $\ZZ[\Omega_n][X]/(X^n-1)$ and $(1, \overline{X}, \dots,
 {\overline{X}}^{\phi(d)-1})$ as a basis for
 $\ZZ[\Omega_n][X]/\Phi_d(X)$ is equal to $A_n$. So we 
have:
\begin{equation}
\label{eq:det-mult}
\det(A_n) = \det(\widetilde{\Psi}_n) = \frac{\det(\rho_2)}{\det(\rho_1)}.
\end{equation}

However, to write the matrix of the $\ZZ[\Omega_n]$-linear maps $\rho_1$ and
$\rho_2$, we need an ordered basis for the codomain of $\rho_1$ which
is  the same as that of $\rho_2$.
So, we may fix any total order $\prec$ on $\Omega_n$ so that if $d_1$ and $d_2$ are
two divisors of $n$ and $d_1 < d_2$, then, all the primitive $d_1$th
roots of unity precede all the primitive $d_2$th roots of unity in the order
$\prec$. Say, 
\begin{equation}
\Omega_n = \coprod_{d\mid n} \{\omega_{d,1}, \dots, \omega_{d,
  \phi(d)}\} = \{\omega_{d,j}: d \mid n,
1\leqslant j \leqslant \phi(d)\}
\end{equation} 
where $\omega_{d,j}$ is a primitive $d$th root of unity ($1\leqslant j
\leqslant \phi(d)$) and $\Omega_n$ is ordered lexicographically. 

The rest of the calculation will determine $\det(\rho_1)$ and
$\det(\rho_2)$. The matrix of $\rho_1$ with respect to the chosen
basis is a block matrix $\diag(\mathcal{D}_1, \dots, \mathcal{D}_d, \dots, \mathcal{D}_n)_{d \mid
  n}$ where $\mathcal{D}_d$ is a  $\phi(d) \times \phi(d)$ matrix of the
following form:
\begin{equation} \label{eq:disc-like}
\mathcal{D}_d = 
\begin{pmatrix}
1 & \omega_{d,1}               & \omega_{d,1}^2 & \cdots & \omega_{d,1}^{\phi(d)-1}\\
1 & \omega_{d,2}               & \omega_{d,2}^2 & \cdots &
\omega_{d,2}^{\phi(d)-1}\\
1 & \omega_{d,3}               & \omega_{d,3}^2 & \cdots &
\omega_{d,3}^{\phi(d)-1} \\
\vdots & \vdots             & \vdots                 &             & \vdots\\
1         &\omega_{d,\phi(d)} & \omega_{d,\phi(d)}^2 & \cdots & \omega_{d,\phi(d)}^{\phi(d)-1}
\end{pmatrix} 
\end{equation}
where $\omega_{d,1} \prec \dots \prec \omega_{d,\phi(d)}$ are primitive $d$th
roots of unity. 

The matrix $[\rho_2]$ of $\rho_2$ with respect to the chosen basis is
a block column matrix \[(R_1, \dots, R_d, \dots, R_n)_{d \mid n}^t\]
where $R_d$ is the following matrix:
\begin{equation}
R_d = \begin{pmatrix}
   1      & \omega_{d,1} & \dots & \omega_{d,1}^{\ell}     & \dots & \omega_{d,1}^{n-1}\\
   \vdots & \vdots  &       & \vdots           &       & \vdots   \\
    1     & \omega_{d,j} & \dots &\omega_{d,j}^{\ell} &\dots& \omega_{d,j}^{n-1}\\ 
   \vdots & \vdots &               & \vdots &                  & \vdots\\
   1      & \omega_{d,\phi(d)} & \dots & \omega_{d,\phi(d)}^{\ell} & \dots & \omega_{d,\phi(d)}^{n-1}
\end{pmatrix}
\end{equation}
Since the matrices $\mathcal{D}_{d}$ and $[\rho_2]$ are
Vandermonde matrices, we have at once that:
\begin{align}
\det(\rho_1) &= \prod_{d \mid n} \det(\mathcal{D}_d)\notag \\
                    &= \prod_{d \mid n}
                   \prod_{1
                      \leqslant i < j \leqslant \phi(d)} (\omega_{d,j}
                    - \omega_{d,i}) \\
\det(\rho_2) &= \prod_{\substack{d_1, d_2 \mid n \\ 1 \leqslant d_1 < d_2
  \leqslant n}}\prod_{\substack{1 \leqslant i \leqslant
    \phi(d_1)\\1 \leqslant j \leqslant \phi(d_2)}} (\omega_{d_2, j} -
  \omega_{d_1, i})\prod_{d \mid n}\prod_{1 \leqslant i < j
  \leqslant \phi(d)} (\omega_{d, j} - \omega_{d, i})
\end{align}
Now, from \eqref{eq:det-mult}, we see that:
\begin{align}
  \det(A_n) &= \prod_{\substack{d_1, d_2 \mid n \\ 1 \leqslant d_1 <
      d_2\leqslant n}} \left(\prod_{\substack{1 \leqslant i \leqslant
    \phi(d_1)\\1 \leqslant j \leqslant \phi(d_2)}} (\omega_{d_2, j} -
  \omega_{d_1, i}) \right) \\
                 &= \prod_{\substack{d_1, d_2 \mid n \\ 1 \leqslant d_1 < d_2
  \leqslant n}} \mathcal{R}(\Phi_{d_2}, \Phi_{d_1}) \label{eq:det-psi-n-formula}
\end{align}
where $\mathcal{R}(f, g)$ is the resultant of the polynomials $f$ and
$g$ (for definition and basic properties of resultants of polynomials,
see \cite[Chapter 1,  Section 3]{VVP}). 
The resultant of pairs of cyclotomic polynomials first appears in print in the work
of E.~Lehmer~\cite[Theorem~4]{Lehmer:1930aa}.  
These resultants were also calculated by Diederichsen \cite[\S3,
Hilfssatz 2]{Dchsen} independently in his work on integral
representations of cyclic groups. We also refer to 
Apostol \cite{TRes} and Dresden \cite{Dresden} for alternative proofs. The following result will
be used to finish off the computation:
\begin{thm}[E.\@ Lehmer]\label{thm:dieder}
Let $m$ and $n$ be positive integers. 
\begin{enumerate} 
\item $\mathcal{R}(\Phi_m, \Phi_n) = 0$ if and only if $m = n$. 
\end{enumerate}
Assume now that $m > n$.
\begin{enumerate}[resume]
\item If $n = 1$, then, 
  \begin{equation}
    \label{eq:res-pair-n=1}
    \mathcal{R}(\Phi_m, \Phi_n) = 
\begin{cases}
 (-1)^{\phi(m)} p &\text{ if } m = p^{\alpha} \text{ for some } \alpha
 > 0\\
 (-1)^{\phi(m)} &\text{ otherwise}
\end{cases}
\end{equation}
\item If $n > 1$ and $\gcd(m, n) = 1$, then, $\mathcal{R}(\Phi_m, \Phi_n)
  = 1$.
\item If $n > 1$ and $\gcd(m, n) > 1$, then, 
\begin{equation}
    \label{eq:res-pair-n>1}
    \mathcal{R}(\Phi_m, \Phi_n) = 
\begin{cases}
 p^{\phi(m)} &\text{ if } \frac{m}{n} = p^{\alpha} \text{ for some }
 \alpha > 0\\
 1 &\text{ otherwise}
\end{cases}
\end{equation}
\end{enumerate}
\end{thm}
We now start from \eqref{eq:det-psi-n-formula} and use
Theorem~\ref{thm:dieder} to get a closed form expression for
$\det(A_n)$ in terms of the prime factorisation of $n$, say, $n =
p_1^{\alpha_1} \cdots \;p_r^{\alpha_r}$.

We first rewrite \eqref{eq:det-psi-n-formula} as follows: 
\begin{align}
  \prod_{\substack{1 \leqslant d_1 < d_2 \leqslant n \\ d_1, d_2 \mid
      n}} \mathcal{R}(\Phi_{d_2}, \Phi_{d_1}) &=
  \prod_{\substack{d \mid n\\d \neq 1}}\mathcal{R}(\Phi_{d}, X-1) \prod_{\substack{1 < d_1 < d_2 \leqslant n \\ d_1, d_2 \mid
      n}} \mathcal{R}(\Phi_{d_2}, \Phi_{d_1}) \label{eq:step-1-split}
  \\
\intertext{As a consequence of
  \eqref{eq:res-pair-n=1}, the first product on the right hand side of 
  \eqref{eq:step-1-split}  evaluates to:} 
\prod_{\substack{d \mid n\\d \neq 1}}\mathcal{R}(\Phi_{d}, X-1) &=
\prod_{\substack{d \mid n\\d \neq 1}} (-1)^{\phi(d)}\Phi_d(1)  =(-1)^{n-1} n.
\intertext{For evaluating the second product, by
  Theorem~\ref{thm:dieder}, note
  that a pair $(d_1, d_2)$ of divisors of $n$ contributes to the
  product if and only if $d_1 \neq 1$ and the ratio $\frac{d_2}{d_1}$ is a prime
  power, say $p_i^{e_i}$ for some $1 \leqslant i \leqslant r$ and
  $1\leqslant e_i \leqslant \alpha_i$; also each such pair contributes
  $p_i^{\phi(d_2)}$ to the product. For a fixed prime $p_i$ ($1
  \leqslant i \leqslant r$) and exponent $e_i$ ($1 \leqslant e_i \leqslant
  \alpha_i$), every divisor $d_1$ of $\frac{n}{p_i^{e_i}}$ with $d_1
  \neq 1$ determines a contributing pair $(d_1, d_2)$ of divisors and
  conversely. Therefore:}
 \prod_{\substack{1 < d_1 < d_2 \leqslant n \\ d_1, d_2 \mid
      n}} \mathcal{R}(\Phi_{d_2}, \Phi_{d_1}) &= \prod_{i=1}^r \prod_{e_i=1}^{\alpha_i} \prod_{\substack{d
    \mid np_i^{-e_i}\\d \neq 1}} p_i^{\phi(d)} \\
& = \frac 1 n\prod_{i=1}^r p_i^{n\sum_{e_i=1}^{\alpha_i}p_i^{-e_i} }
\end{align}
This finishes the computation and we now have:
\begin{thm}
For a positive integer $n$, we have:
\begin{align}
\det(A_n) &=  \prod_{\substack{1 \leqslant d_1 < d_2 \leqslant n \\ d_1, d_2 \mid
      n}} \mathcal{R}(\Phi_{d_2}, \Phi_{d_1})\label{eq:det-in-terms-of-res}\\
&= (-1)^{n-1} \prod_{i=1}^r p_i^{\frac {n(1-p_i^{-\alpha_i})} {(p_i-1)}}.
\end{align}
\end{thm}
Following the computations done before Theorem~\ref{thm:dieder}, we
may prove: 
\begin{thm}
\label{thm:general-result}
Suppose that $f$ is a monic polynomial over an integral domain $R$ and 
\[f = \prod_{i=1}^n f_i\]
is a factorisation of $f$ into monic
polynomials. Then, the determinant of the canonical map $\Psi_f$
written with respect to the standard basis is:
\begin{equation}
  \label{eq:gen-det}
  \det(\Psi_f) = \prod_{1 \leqslant i < j \leqslant n} \mathcal{R}(f_j, f_i).
\end{equation}
\end{thm}
\begin{proof}
Let $f_1, \dots, f_n$ and $f$ be as given. The argument outlined above
shows that \eqref{eq:gen-det} holds
when $f$ has no repeated roots. 

To handle the general case, we need an argument using the Zariski topology. Suppose that $f_j$ has degree $d_j$ and put $d = \sum_{j=1}^n d_j$.
Recall that a generic monic polynomial $P_j$ of degree $d_j$ is given by 
 \begin{equation}
   P_j(X) = \sum_{k=0}^{d_j - 1} T_j(k) X^k + X^{d_j}
 \end{equation}
with coefficients in the polynomial algebra \[A = \mathbf{Z}[\{T_j(k):
0 \leqslant k \leqslant d_j - 1\}_{1 \leqslant j \leqslant n}].\] 

Put $P = \prod_{j=1}^n P_j$. We consider the natural map 
\[\Psi_P: A[X]/(P) \to \bigoplus_{j=1}^n A[X]/(P_j).\]
Let us note the following:
\begin{itemize}
\item 
The map $\Psi_P$ is $A$-linear and consequently, $\det(\Psi_P)$ belongs to $A$. 
\item For $1 \leqslant i_1, i_2 \leqslant r$, the resultant $\mathcal{R}(P_{i_1}, P_{i_2})$ is also an element of $A$. 
\end{itemize}
We begin with the following observation: 
\begin{obsn*}
  Let $B$ be any ring. Given any element $\underline{c} \in B^d$, there is a unique map from
$h_{\underline{c}}: A \to B$ such that $(h_{\underline{c}}(T_j(k)):0 \leqslant k
\leqslant d_j - 1, 1 \leqslant j \leqslant n)=\underline{c}$.
\end{obsn*}
We claim that $\det(\Psi_P) = \prod_{1 \leqslant i < j \leqslant n}
\mathcal{R}(P_j, P_i)$ as polynomials over integers. Indeed, the equality holds on the set
$S := \{\underline{c} \in \mathbf{Z}^d: h_{\underline{c}}(\Disc(P)) \neq
0\}$ where $\Disc(P)$ is the discriminant of $P$ given by the
square of the determinant
of the Vandermonde matrix corresponding to the roots of $P$ in an
algebraic closure of $\mathbf{Q}$ (cf.~\eqref{eq:disc-like}; see also \cite[IV,
\S6, p.192 and IV, \S8, p.204]{LangAlg2002}).  
The set $S$ is non-empty (take $P(X) = X(X-1) \dots (X - d + 1)$ for
example) and open. Since $\mathbf{Z}^d$ is irreducible for the Zariski
topology, we have that $S$ is dense in $\mathbf{Z}^n$ so that our
claim follows. 

Now, to prove \eqref{eq:gen-det}, we need to only note the following: if
$\underline{c} \in R^d$ is such that $h_{\underline{c}}(P_j) = f_j$
for all $1 \leqslant j \leqslant n$, then: 
\begin{align*}
\det(\Psi_f) &= h_{\underline{c}}(\det(\Psi_P)) \\
  &= h_{\underline{c}}\left(\prod_{1 \leqslant i < j \leqslant n}
\mathcal{R}(P_j, P_i)\right) \\
  &= \prod_{1 \leqslant i < j \leqslant n}
\mathcal{R}(f_j, f_i)
\end{align*}
and this completes the proof. 
\end{proof}
\begin{rem}
  From the above theorem, the resultant $\mathcal{R}(f_2, f_1)$ is the determinant of
the map $\Psi_{f_1f_2}$ written with respect to the standard
basis. Specialising to $f_1 = \Phi_m(X)$ and $f_2 = \Phi_n(X)$ with $m
> n$, we have the following exact sequence:
\[
\begin{tikzcd}
  0 \ar{r}& \frac{\ZZ[X]}{\Phi_m(X)\Phi_n(X)} \ar{r}& \frac{\ZZ[X]}{\Phi_m(X)} \oplus
  \frac{\ZZ[X]}{\Phi_n(X)} \ar{r} & G(\Phi_m(X)\Phi_n(X)) \ar{r} & 0
\end{tikzcd}
\]
Now, we have $|G(f)| = |\mathcal{R}(\Phi_m, \Phi_n)|$. Also, the
intersection of the ideal generated by $\Phi_m$ and $\Phi_n$ with
$\ZZ$ is given by (\cite[Theorem~2]{Dresden}):
\begin{equation}
  \label{eq:ideal-intersection}
  \langle \Phi_m(X), \Phi_n(X)\rangle \cap \ZZ= \begin{cases} 
p\ZZ &\text{ if } \frac{m}{n} = p^\alpha \text{ for some } \alpha > 0\\
\ZZ   &\text{otherwise} 
\end{cases} 
\end{equation}
Now, setting $r = \phi(m) + \phi(n)$, if $(e_0, \dots, e_{r-1})$ are
elementary divisors of $\Psi_{\Phi_m\Phi_n}$, then, $e_i \mid p$ and $\prod e_i =
|\mathcal{R}(\Phi_m, \Phi_n)|$. If $\frac{m}{n}$ is not power of a
prime, then, $e_i =1$ for all $i$. If $\frac{m}{n}$ is power of a
prime $p$, the elementary divisors are given by $e_i = 1$ for $0
\leqslant i \leqslant \phi(n)-1$ and $e_j =
p$ for $\phi(n)\leqslant  j \leqslant \phi(m)+\phi(n)-1$. 
\end{rem}
\begin{rem}
  Let $S$ be a UFD. Then, $R = S[x_1, \dots, x_n]$ is a UFD and considering
  the polynomial $f(X) = \prod_i (X -x_i)$ in
  Theorem~\ref{thm:general-result}, we see that Vandermonde
  determinant falls out as a special case. The elementary divisors of
  a Vandermonde matrix over a Dedekind domain has been calculated by
  M.\@ Bhargava \cite[Lemma 2]{MBGFaDD}.
\end{rem}
\subsection*{Acknowledgements}
The authors would like to thank the Institute of
Mathematical Sciences, Chennai and Indian
Statistical Institute, Bangalore, where various parts of this
work was carried out, for their hospitality and
support. The authors thank  Amritanshu Prasad for suggesting numerous improvements. The authors are grateful to
Vikram Sharma for guiding them with references for the write-up in 
Section~\ref{ssec:time-complexity}.

\bibliographystyle{elsarticle-num} 
\bibliography{refs_for_smith}
\end{document}